\documentclass[a4paper,11pt]{article}

\usepackage{amssymb,amscd,amsthm,amsmath,color,fancyhdr,mathrsfs}
\usepackage{enumerate,enumitem}
\usepackage{graphicx}
\usepackage[all]{xy}
\usepackage{authblk}
\usepackage{geometry}
\usepackage{graphicx} % Required for inserting images
\usepackage{mathtools}
\usepackage{float}
\usepackage{lineno}
\usepackage{hyperref}
\hypersetup{colorlinks,linkcolor={blue},citecolor={blue},urlcolor={red}}
\usepackage{palatino}

\usepackage{easyReview}

\newtheorem{thm}{Theorem}[section]
\newtheorem{prop}[thm]{Proposition}
\newtheorem{lem}[thm]{Lemma}
 
\newtheorem{cor}[thm]{Corollary}

\theoremstyle{definition}
\newtheorem{defn}[thm]{Definition} 

\def\wtilde{\widetilde}
\def\inc{\mathbin{\widetilde\in}}
\def\incs{\mathbin{\widetilde\subset}}

\def\cIm{\operatorname{cIm}}

\def\cGr{\mathsf{cGr}}

\def\Sets{\mathsf{Sets}}

\def\cKer{\textbf{cKer}}

\def\Ker{\operatorname{Ker}}
\def\cKer{\operatorname{cKer}}
\def\NSCM/(A,B,\mu){\mathsf{NSCM/(A,B,\mu)}}
\def\NSGGd/G{\mathsf{NSGGd/G}}

\def\epsilon{\varepsilon}
\def\wtilde{\widetilde}

\numberwithin{equation}{section}

\title{\bf From cssc-crossed modules to categorical groups}

\author[1]{Tamar Datuashvili\thanks{\textbf{Corresponding author:} Tamar Datuashvili (e-mail: tamar.datu@gmail.com)}}
\affil[1]{\small{A. Razmadze Mathematical Institute of I. Javakhishvili Tbilisi State University, 2, M. Aleksidze II Lane, 0193, Tbilisi, GEORGIA}}
\author[2]{Osman Mucuk\thanks{O. Mucuk (e-mail : mucuk@erciyes.edu.tr)}}
\affil[2]{\small{Department of Mathematics, Erciyes University, 38039, Kayseri, TURKEY}}
\author[2]{Nazmiye Alemdar\thanks{N. Alemdar (e-mail : nakari@erciyes.edu.tr)}}
\author[3]{Tun\c{c}ar \c{S}ahan\thanks{T. \c{S}ahan (e-mail : tuncarsahan@gmail.com)}}
\affil[3]{\small{Department of Mathematics, Aksaray University, 68100, Aksaray, TURKEY}}

\makeatletter
\let\@@citation@@=\citation

\renewcommand{\citation}[1]{\@@citation@@{#1}%
	\@for\@tempa:=#1\do{\@ifundefined{cit@\@tempa}%
		{\global\@namedef{cit@\@tempa}{}}{}}%
}

\def\@lbibitem[#1]#2#3\par{%
	\@ifundefined{cit@#2}{}{\item[\@biblabel{#1}\hfill]}%
	\if@filesw
	{\let\protect\noexpand
		\immediate
		\write\@auxout{\string\bibcite{#2}{#1}}}\fi\ignorespaces
	\@ifundefined{cit@#2}{}{#3}}
\def\@bibitem#1#2\par{%
	\@ifundefined{cit@#1}{}{\item}%
	\if@filesw \immediate\write\@auxout
	{\string\bibcite{#1}{\the\value{\@listctr}}}\fi\ignorespaces
	\@ifundefined{cit@#1}{}{#2}}
\makeatother

\begin{document}
%\linenumbers	
\date{}
\maketitle

\begin{abstract}
For any cssc-crossed module a category is constructed, equipped with a structure and proved that this is a coherent categorical group. Together with a result of the previous paper, where to any categorical group the cssc-crossed module is associated, this construction will enable us to prove an equivalence between the categories of categorical groups and of cssc-crossed modules in the sequel to this paper.
\end{abstract}	
	
\section{Introduction} 
\label{intro}
	
Our aim was to obtain for categorical groups an analogous description in terms of certain crossed module type objects as we have for $\mathcal{G}$-groupoids  obtained by Brown and Spencer \cite{BS1}, which are strict categorical groups, or equivalently, group-groupoids or internal categories in the category of groups. By a categorical group we mean a coherent 2-group in the sense of Baez and Lauda \cite{baez-lauda-2-groups}. 
It is important to note, that it is a well known fact that a categorical group is equivalent to a strict categorical group \cite{baez-lauda-2-groups,JoyalStreet,Sinh 2}, but we do not have an equivalence  between the corresponding categories.
%It is important to note that it is well known that a categorical group is equivalent to a strict categorical group \cite{baez-lauda-2-groups,JoyalStreet,Sinh 2}, but we do not have an equivalence  between the corresponding categories. 
This idea brought us to a new notion of group up to congruence relation introduced in \cite{Dat2020}. In this way we came to the definition of c-group and the corresponding category. Then we defined actions in this category and introduced the notion of c-crossed module. Among this kind of objects we distinguished connected, strict  and special c-crossed modules denoted as cssc-crossed modules \cite{Dat2020}. We proved that any categorical group $\mathbb{C}$ gives rise to a cssc-crossed module. In the presented paper for any cssc-crossed module $\partial\colon M \rightarrow N$ we construct a categorical group $\mathbb{G}$. A set of objects of $\mathbb{G}$ is defined in the standard way, as it is in the case of group-groupoids. The definition of a set of arrows turned out to be complicated. We consider the product $N\times M$; then enlarge this set with special congruences from the c-crossed module $\partial\colon M \rightarrow N$. After this we introduce two types of identifications in the obtained set and define an operation of addition there, as well as the composition of arrows, the identity arrow of each object in $\mathbb{G}$ and the inverse arrow of each arrow. We prove that all these operations do not depend on the choice of representatives. After all we prove that the constructed category is a coherent categorical group.

We plan to show that there is an equivalence between the category of coherent categorical groups  and the category of cssc-crossed modules. We hope that this result will give a chance to consider for categorical groups problems analogous to those considered and solved in the cases of strict categorical groups in terms of group-groupoids and internal categories in \cite{BM1,Dat1,Dat2,Dat3,Dat4,Mucuk2015}.

\section{Preliminary results and definitions}

Since the paper is a continuation of the previous one, in this section as a preparation we summarise some preliminaries from \cite{Dat2020}.

%Since the paper is a continuation of the previous paper, in this  section we remind some necessarily preliminaries from \cite{Dat2020}. 

\subsection{Categorical groups}

The following definition of a monoidal category goes back to Mac Lane \cite{Mac71}.

A \textbf{monoidal category} is a category $\mathsf{C}=(C_0, C_1, d_0, d_1, i, m)$ equipped with a bifunctor $+\colon\mathsf{C}\times\mathsf{C}\rightarrow\mathsf{C}$ called the monoidal sum, an object $0$ called the zero object, and the family of three natural isomorphisms  $\alpha=\alpha_{x,y,z}\colon (x+y)+z\stackrel{\approx}{\rightarrow} x+(y+z)$, $\lambda_x\colon 0+x\stackrel{\approx}{\rightarrow} x$ and $\rho_x\colon x+0\stackrel{\approx}{\rightarrow} x$, for all $x,y,z\in C_0$, such that the following diagrams commute for all $x,y,z,t\in C_0$.
		\[\xymatrix{
			((x+y)+z)+t  \ar[d]_{\alpha+1} \ar[r]^-{\alpha} & (x+y)+(z+t)  \ar[r]^-{\alpha} & x+(y+(z+t)) \\
			(x+(y+z))+t  \ar[rr]_-\alpha & & x+((y+z)+t) \ar[u]_{1+\alpha} }\]
		\[\xymatrix{
			(x+0)+y  \ar[r]^-{\alpha} \ar[d]_-{\rho+1} &  x+(0+y)  \ar[d]^-{1+\lambda} \\
			x+y  \ar@{=}[r] & x+y }\]
Moreover, all diagrams involving $\alpha,\lambda$, and $\rho$ must commute. A monoidal groupoid is a monoidal category in which every morphism is invertible. 

In this definition we use the term monoidal sum and denote it as $+$, instead of monoidal product, used in the original definition, and write the operation additively.
	
It follows from the definition that $1_0+f\approx f+1_0 \approx f$, for any morphism $f$. A monoidal category is said to be strict if the natural isomorphisms $\alpha, \lambda$, and $\rho$ are identities.
	
For any two monoidal categories $\mathsf{C}=(C,+,0,\alpha,\lambda,\rho)$ and $\mathsf{C'}=(C',+',0',\alpha',\lambda',\rho')$, a functor $T:C\rightarrow C'$ satisfying $T(x+y)=Tx+'Ty$, $T(f+g)=Tf+'Tg, T0=0'$, $T\alpha_{x,y,z}=\alpha'_{Tx,Ty,Tz}$, $T\lambda_x=\lambda'_{Tx}$, $T\rho_x=\rho'_{Tx}$ for all objects $x,y,z$ and morphisms $f$ and $g$ is called a (strict) morphism of monoidal categories. \[T:(C,+,0,\alpha,\lambda,\rho)\rightarrow (C',+',0',\alpha',\lambda',\rho')\]

Let $x$ be an object in a monoidal category $\mathsf{C}$. If there is an object $y\in C_0$ such that $x+y\approx 0$ and $y+x\approx 0$ then $y$ is called an inverse of $x$. It is a well known fact that if any object has a one-sided inverse in a monoidal category, then any object is invertible \cite{baez-lauda-2-groups,JoyalStreet}.
	
A monoidal groupoid $\mathsf{C}=(C_0, C_1, d_0, d_1, i, m)$ is called a \textit{categorical group} if for every object $x\in C_0$ there is an object $-x\in C_0$ with a family of natural isomorphisms	$\epsilon_x\colon -x+x\approx 0$ and $\delta_x\colon x+(-x)\approx 0$ such that the following diagrams are commutative:
		\[\xymatrix@C=2.3pc{
			0 + x \ar[r]^<<<<<<{\delta^{-1}_x + 1} \ar[d]_{\lambda _x}
			& (x + (-x)) + x \ar[r]^<<<<<{\alpha_{x, -x ,x}}
			& x+( -x + x) \ar[d]^{1_x + \epsilon_x} \\
			x \ar[rr]_{\rho ^{-1}_x}
			& & x + 0  }
		\]
		
		\[
		\xymatrix@C2.3pc{
			-x + 0
			\ar[r]^<<<<<<{1 + \delta^{-1}_x}
			\ar[d]_{\rho _{-x}}
			& -x + (x + (-x))
			\ar[r]^{\alpha^{-1}_{-x, x, -x}}
			& (-x + x)+ (-x)
			\ar[d]^{\epsilon_x + 1_{-x}} \\
			-x
			\ar[rr]_{\lambda^{-1}_{-x}}
			&& 0 + (-x)   }
		\]
%$(1_x + \epsilon_x)\alpha_{x, -x ,x}(\delta^{-1}_x + 1)=\rho ^{-1}_x\lambda _x$ and $(\epsilon_x + 1_{-x})\alpha^{-1}_{-x, x, -x}(1 + \delta^{-1}_x)=\lambda^{-1}_{-x}\rho_{-x}$ 
	
It is important and  a well-known fact that the definition of a categorical group implies that for any morphism $f:x\rightarrow x' \in C_1$ there is a morphism $-f\colon -x\rightarrow -x'$ with natural isomorphisms $-f+f\approx 0$ and $f+ (-f)\approx 0$, where the morphism $0$ is $1_0$ (see e.g. \cite{Sinh 1}).
The natural isomorphisms $\alpha,\lambda$, $\rho, \epsilon, \delta$, identity transformation $1_ \mathsf{C}\rightarrow 1_\mathsf{C}$, their compositions and sums are called \emph{special isomorphisms} in \cite{Dat2020}. A categorical group is said to be coherent if all diagrams involving special isomorphisms commute \cite{Laplaza,baez-lauda-2-groups}. The categorical group defined above is coherent. See \cite[Chapter VII Section 2]{Mac71} for coherence of monoidal categories.
	
The functorial properties of addition $+$ implies that in a categorical group we have $-1_x=1_{-x}$, for any $x\in C_0$. Since an isomorphism between morphisms $\theta :f\thickapprox g$ means that there exist isomorphisms $\theta_i:d_i(f)\rightarrow d_i(g), i=0,1$ with $\theta_1 f=g\theta_0$, the naturality property of special isomorphisms implies that there exist special isomorphisms between the morphisms in $C_1$. But if $\theta_i, i=0,1$ are special isomorphisms, it does not imply that $\theta$ is a special isomorphism; in this case $\theta$ is called a \emph{weak special isomorphism}. It is obvious that a special isomorphism between the morphisms in $C_1$ is a weak special isomorphism. Note that if $f\approx f'$ is a weak special isomorphism, then the coherence property implies that $f'$ is the unique morphism weakly specially isomorphic to $f$ with the same domain and codomain objects as $f'$.
	
In \cite{Dat2020}, we defined (strict) morphisms between categorical groups, which satisfy conditions of (strict) morphisms of monoidal categories. Note that this definition implies: $T(-x)=-{T(x)}$ and $T(-f)=-{T(f)}$, for any object $x$ and arrow $f$ in a categorical group. Categorical groups form a category with (strict) morphisms between them. 	For any categorical group $\mathsf{C}=(C_0, C_1, d_0, d_1, i, m)$ denote $\Ker d_0=\{f\in C_1\mid d_0(f)\approx 0\}$  and $\Ker d_1=\{f\in C_1\mid d_1(f)\approx 0\}$.
	
\begin{lem}\label{comm}
Let $\mathsf{C}=(C_0, C_1, d_0, d_1, i, m)$ be a categorical group. For any $f\in \Ker d_1$ and $g\in \Ker d_0$ we have a weak special isomorphism $f+g\approx g+f$.
\end{lem}

\subsection{Groups up to congruence relation}
	
Now we recall the definition of a group up to congruence relation or briefly a c-group. Let the pair $X_R$ denotes a set $X$ with an equivalence relation $R\subseteq X\times X$. These kind of objects form a category, denoted by $\wtilde{\Sets}$, where the morphisms are functions $f\colon X_R\rightarrow Y_S$, such that $f(x)\sim_S f(y)$, whenever $x\sim_R y$. 

Product in this category consists of the cartesian product of the sets and the usual product of the equivalence relations.
	
\begin{defn}\label{defcgr}
A \textit{c-group} is an object $G_R$ in $\wtilde{\Sets}$ with a morphism $\mathfrak{m}\in \wtilde{\Sets}((G\times G)_{R\times R},G_R)$, denoted by $\mathfrak{m}(a,b)=a+b$, for any $a,b\in G$, satisfying the following conditions:
%
%		Let $G_R$ be an object in $\wtilde{\Sets}$ and
%		\[ \begin{array}{cccl}
%			m \colon &G\times G& \longrightarrow & G\\
%			&(a,b)    & \longmapsto     & a+b
%		\end{array}\]
%		a morphism in $\wtilde{\Sets}$, i.e.,  $m\in \wtilde{\Sets}((G\times G)_{R\times R},G_R)$. $G_R$ is called a {\em c-group} if the following axioms are satisfied.
\begin{enumerate}[label=(\roman{*}), leftmargin=1cm]
	\item\label{def:cgri} $a+(b+c)\sim_R(a+b)+c$, for all $a,b,c\in G$;
	\item\label{def:cgrii}  there exists an element $0\in G$ such that $a+0\sim_R a\sim_R 0+a$, for all $a\in G;$
	\item\label{def:cgriii} for each $a\in G$ there exists an element $-a$ such that  $a+(-a)\sim_R 0$ and $-a+a\sim_R0$.
\end{enumerate}
\end{defn}
	
In a c-group $G_R$, the given element $0\in G$ is called  {\em zero element} and for any $a\in G$ the given element $-a\in G$ is called the {\em inverse} of $a$. The congruences involved in the conditions \ref{def:cgri}--\ref{def:cgriii} of the definition, their compositions and sums are called \emph{special congruences}.
	
Every group $G$ can be considered as a c-group with $R=\{(a,a)\mid a\in G\}$, i.e. $R$ is the equality $(=)$ relation. See \cite{Dat2020} for the properties and examples of c-groups.
	
For any two c-groups $G_R$ and  $H_S$, a morphism $f\colon G_R\rightarrow H_S$ in $\wtilde{\Sets}$ is called a \textit{c-group morphism} if $f(a+b)=f(a)+f(b)$,  for any $a,b\in G$. It is easy to see that a morphism between c-groups preserves congruences between elements; moreover $f(0)\sim 0$ and $f(-a)\sim -f(a)$, for any $a\in G$. Hence, a morphism between c-groups carries special congruences to special congruences between pairs of elements.
	
%A category can be constructed with objects as c-groups and with morphisms as c-group morphisms. This category is denoted by $\cGr$.

A category with objects as c-groups, with morphisms as c-group morphisms, with obvious composition of morphisms and identity morphisms of each object will be denoted by $\cGr$.

For a c-group morphism $f\colon G_R\rightarrow H_S$, the subset $\cKer f=\{a\in G_R \mid f(a)\sim_S 0_H \}$ of $G_R$ is called the \textit{c-kernel} and the subset $\cIm f=\{b\in H_S \mid \exists a\in G_R,  f(a)\sim_S b\}$ of $H_S$ is said to be the \textit{c-image} of $f$.

Let $H$ be a subset of the underlying set $G$ of a c-group $G_R$ and let $S=R\cap(H\times H)$. In this case, it is easy to see that $S$ is an equivalence relation on $H$. If $H_S$ is a c-group with the operation induced from $G_R$ then $H_S$ is called a c-subgroup of $G_R$. Note that $\cKer f$ is a c-subgroup of $G_R$. In particular, $\cKer d_0$, for a categorical group $\mathsf{C}=(C_0, C_1, d_0, d_1, i, m)$, is a c-subgroup of $C_1$ with the congruence relation on $\cKer d_0$ induced by the isomorphisms in $C_1$.

Let $G_R$ be a c-group and let $H$ be a subset of $G$. If for an element $a\in G$ there exists an element $b\in H$ such that $a\sim_R b$ then we write $a\inc H$. If $H$ and $H'$ are two subsets of $G_R$, then we write $H\incs H'$ if for any $h\in H$ we have $h \inc H'$. If $H\incs H'$ and $H'\incs H$, then we write $H\sim H'$.
	
\begin{defn}
Let $H_S$ be a c-subgroup of a c-group $G_R$. Then $H_S$ is called
\begin{enumerate}[label=(\roman{*}), leftmargin=1cm]
	\item \textit{normal} if $g+h-g\inc H_S$, for any $h\in H_S$ and $g\in G$; 
	\item \textit{perfect} if $g\inc H$ implies $g\in H$, for any $g\in G$.
\end{enumerate}
\end{defn}
	
One can see that, for a c-group morphism $f\colon G_R\rightarrow H_S$, $\cKer f$ is a perfect and normal c-subgroup of $G_R$, and $\cIm f$ is a perfect c-subgroup of $H_S$.
	
\begin{defn}
A c-group $G_R$ is called \emph{connected} if $g\sim g'$ for any $ g, g'\in G$.
\end{defn}
	
%Quotient object $G/H$ of c-groups where $H$ is a normal c-subgroup of a c-group $G$ is described in \cite{Dat2020} as in the following way: Consider the classes  $\{g+H\mid g\in G\}$. If $(g+H) \cap (g'+H)\neq \emptyset$, then we obtain $-g+g'\inc H$, which implies that $g+H\sim g'+H$. Now consider the set of these classes $\{\mathrm{cl}(g+H)\mid g\in G\}$, where $\mathrm{cl}(g+H)=\cup\{x\in G\mid x\inc g+H\}$. Define $G/H=\{\mathrm{cl}(g+H)\mid g\in G\}$. An addition operation in this set is defined by $\mathrm{cl}(g+H)+\mathrm{cl}(g'+H)=\mathrm{cl}((g+g')+H)$, for any $g,g'\in G$. This operation is well defined, it is associative and we have the unit element $\mathrm{cl}(0+H)$. Actually the constructed object is a group, the congruence  relation on $G/H$ is the equality $``="$. We have an obvious surjective morphism $p: G \rightarrow G/H$.
%	
%	\begin{lem}\cite{Dat2020} (i) If $G$ is a c-group and $H$ is a normal c-subgroup in $G$, then for any group $G'$ and c-group morphism $f:G\rightarrow G'$, if $f(h)=0$ for any $h\in H$, there exists a unique morphism $\theta :G/H \rightarrow G'$, in $\cGr$  such that $\theta p=f$.\newline
%		(ii)  If $H$ is a perfect normal c-subgroup in $G$, then $H=\cKer p$.
%	\end{lem}

\subsection{Actions and crossed modules in $\cGr$}

From now on we omit congruence relation symbols for c-groups $A$ and $B$ if no confusion arise.

\begin{defn}
An (left) \textit{action} of a c-group $B$ on a c-group $A$ is a function $B\times A\rightarrow A$ denoted by $(b,a)\mapsto b\cdot a$ which satisfies the following conditions
\begin{enumerate}[label=(\roman{*}), leftmargin=1cm]
	\item $b\cdot(a+a_1)\sim (b\cdot a)+(b\cdot a_1)$,
	\item $(b+b_1)\cdot a\sim b\cdot(b_1\cdot a)$,
	\item $0\cdot a\sim a$,
	\item If $a\sim a_1$ and $b\sim b_1$ then $b\cdot a \sim b_1\cdot a_1$,
\end{enumerate}
for $a,a_1\in A$ and $b,b_1\in B$.
\end{defn}

A semi-direct product is defined as follow: Let $A, B \in \cGr$ and suppose that $ B $ acts on $ A $ satisfying the conditions (i)--(iv). Then the product $B \times A$ in  $\cGr$ becomes a c-group with the operation $(b',a')+(b,a)=(b'+b, a'+b'\cdot a)$ for any $b, b'\in B, a, a'\in A$ where the congruence relation is the product relation, i.e. $(b,a)\sim (b',a')$ if and only if $b\sim b'$ and $a\sim a'$. Here, $(0,0)$ is a zero element in $B \times A$ and $(-b, -b\cdot (- a))$ is the opposite element of the pair $(b,a)\in B \times A$. This c-group is called the semi-direct product $B \ltimes A$ in $\cGr$.
	
For any morphism $f:D\rightarrow D'$ in $\cGr$, $f$ is called an isomorphism up to congruence relation or briefly a c-isomorphism if there is a morphism $f':D'\rightarrow D$, such that $ff'\sim 1_{D'}$ and $f'f\sim 1_D$. This kind of an isomorphism is denoted by $\tilde{\approx}$, i.e. $f\colon D \mathrel{\tilde{\approx}} D'$.
	
For a semi-direct product object $B \ltimes A$ in $\cGr$, we have a natural projection $p': B \ltimes A \rightarrow B$. In this case, there is a c-isomorphism $\cKer p' \mathrel{\tilde{\approx}} A$ which need not to be an isomorphism as in the case of groups.

From a categorical group $\mathsf{C}=(C_0, C_1, d_0, d_1, i, m)$, we obtain a split extension
\begin{equation} \label{extension}
	\xymatrix{0\ar[r]&\cKer d_0\ar[r]^-{j}&C_1\ar[r]_-{d_0}&C_0\ar[r] \ar@/_/[l]_-{i}&0}
\end{equation}
where $i\colon C_0\rightarrow C_1$ is a section of $d_0$. Thus we define an action of $C_0$ on $\cKer d_0$ by
\[ \begin{array}{rcl}
	C_0\times \cKer d_0& \longrightarrow & \cKer d_0,\\
	(r,c)              & \longmapsto     & r\cdot c=i(r)+(j(c)-i(r)).
\end{array}\]

\begin{prop}
	The action of $C_0$ on $\cKer d_0$ satisfies the conditions for an action in $\cGr$.
\end{prop}
	
\begin{defn}
Let $G$ and $H$ be two c-groups, let $\partial\colon G\rightarrow H$ be a morphism of c-groups and let $H$ act on $G$. We call $(G,H,\partial)$ a \emph{c-crossed module} if the following conditions are satisfied:
\begin{enumerate}[label=(\roman{*}), leftmargin=1cm]
	\item $\partial(b\cdot a)= b+(\partial(a)-b)$,
	\item $\partial(a)\cdot a_1\sim a+(a_1-a)$,
\end{enumerate}
for $a,a_1\in G$ and $b\in H$.
\end{defn}
	
Let $(G,H,\partial)$ and $(G',H',\partial')$ be two c-crossed modules. A \emph{c-crossed module morphism} is a pair of morphisms $\langle f,g\rangle\colon(G,H,\partial)\rightarrow(G',H',\partial')$ such that $g\partial=\partial'f$ and $f(b\cdot a)= g(b)\cdot f(a)$, for all $b\in H$ and $a\in G$, where $f$ and $g$ are morphisms of c-groups. c-crossed modules and morphisms of c-crossed modules form a category.
	
Let $H$ be a normal c-subgroup of a c-group $G$. One can easily see that, in general, we do not have a usual action by conjugation of $G$ on $H$. However, we have a similar situation as given below.
	
\begin{lem}
If $H$ is a perfect normal c-subgroup of a c-group $G$, then we have an action of $G$ on $H$ in the category $\cGr$ and the inclusion morphism defines a c-crossed module.
\end{lem}
	
\begin{defn}\label{connected cr}
A c-crossed module $(G,H,\partial)$ is called \textit{connected} if $G$ is a connected c-group.
\end{defn}

\begin{prop}
For a categorical group $\mathsf{C}=(C_0, C_1, d_0, d_1, i, m)$, $(\cKer d_0,C_0,d)$ is a connected c-crossed module where $d={d_1}|_{\cKer d_0}$.
\end{prop}

In \cite{Dat2020} we introduced another c-group, $\textbf{Star}_{\mathsf{C}} 0$,  from a categorical group $\mathsf{C}=(C_0, C_1, d_0, d_1, i, m)$ apart from $\cKer d_0$. As a set $\textbf{Star}_{\mathsf{C}} 0=\{f\in C_1\mid  d_0(f)=0\}$ and the addition operation on $\textbf{Star}_{\mathsf{C}} 0$ is given by $f+f'=(f+f')\gamma$, where $f+f':0+0\rightarrow d_1(f)+d_1(f')$ is a sum in $C_1$ and $\gamma : 0 \rightarrow 0+0$ is the unique special isomorphism in $C_1$. The congruence relation on $\textbf{Star}_{\mathsf{C}} 0$ is induced by the relation on $C_1$, which is the relation of being isomorphic in $C_1$. $C_0$ is also a c-group where the congruence relation is given by isomorphisms between the objects. Moreover, there is an action of $C_0$ on $\textbf{Star}_{\mathsf{C}} 0$ given by $r\cdot c=(i(r)+(c-i(r)))\gamma$, for any $r\in C_0, c\in \textbf{Star}_{\mathsf{C}} 0$, where $\gamma\colon 0\approx r+(0-r)$ is a special isomorphism.

\begin{defn}\label{strict cr}
A c-crossed module $(G,H,\partial)$ is called \textit{strict} if,
\begin{enumerate}[label=(\roman{*}), leftmargin=1cm]
	\item $\partial(b\cdot a)= b+(\partial(a)-b)$,
	\item $\partial(a)\cdot a_1= a+(a_1-a)$,
\end{enumerate}
for $a,a_1\in G$ and $b\in H$.
\end{defn}
	
\begin{defn}\label{weak special} 
In a c-crossed module $(G,H,\partial)$ a congruence $g\sim g'$ in $G$ is called a \textit{weak special congruence} if $\partial (g)\sim\partial(g')$ is a special congruence in $H$.
\end{defn}

For a c-crossed module $(G,H,\partial)$ every special congruence in $G$ is a weak special congruence since the morphism $\partial$ carries any special congruence to a special congruence between pairs of elements
	
\begin{defn}\label{special}
A c-crossed module $(G,H,\partial)$ is called \emph{special} if for any congruence $\gamma :\partial(c)\sim r$,  there exists $c'\sim c$, such that $\partial(c') =r$, where $c, c'\in G$ and $r\in H$. If $\gamma $ is a special congruence, then  $c'$ is the unique element in $G$ which is special weakly congruent to $c$.
\end{defn}
	
A c-crossed module is called as a \textit{cssc-crossed module} if it is connected, strict, and special c-crossed module. This kind of crossed module is exactly that we were looking for.
	
\begin{thm}
For a categorical group $\mathsf{C}=(C_0, C_1, d_0, d_1, i, m)$ the triple $(\textbf{\em Star}_{\mathsf{C}} 0, C_0, d)$ is a cssc-crossed module, where $d={d_1}|_{\textbf{\em Star}_{\mathsf{C}}}$.
\end{thm}

\section{Objects and arrows in $\mathbb{G}$}
	
In this section we start with a cssc-crossed module and construct the objects and arrows of a category associated with this cssc-crossed module.
	
Let $\partial: M \rightarrow N$ be a cssc-crossed module. Define a category $\mathbb{G}$ with the set of objects $\mathbb{G}_0=N$, which has a c-group structure according to the definition of c-crossed module. Consider the product $N \times M$. Elements of $N \times M$ are the pairs $(r,c)$, where $r \in N$ and $c \in M$. As we know from the definition of a c-group there exist special congruences between certain elements in $N$. We will denote these congruences by $\alpha, \beta,...$. If $\alpha:r \sim  r' $ is a special congruence in $N$, then we write an arrow $\alpha: r \rightarrow  r' $ with $\mathsf{dom}(\alpha)=r$, $\mathsf{codom}(\alpha)=r'$. Since for any special congruence $\alpha: r \sim r' $, there exists the special congruence ${{\alpha}^{-1}}: r' \sim  r $ and the corresponding arrows in $N \times M$ are called by special isomorphisms.
	
We now enlarge the set $N\times M$ as follows: Consider the set $N \bar{\times} M$ with elements of the type $\alpha, \beta,...,\beta (r,c) \alpha $, where $\alpha, \beta,...$ are all special isomorphisms in  $N \times M$, $(r,c) \in N \times M$, and in $\beta (r,c) \alpha$ we have $\mathsf{codom}(\alpha)=r$, $\mathsf{dom}(\beta)=\partial(c)+r$. From this definition it follows that any element $(r,c) \in N \times M$ is an element in $N \bar{\times} M$, where $\alpha=1_r: r \xrightarrow{=} r$ and $\beta=1_{\partial(c)+r}: \partial(c)+r \xrightarrow{=} \partial(c)+r$ are special identity arrows. We define $\mathsf{dom}(\beta (r,c) \alpha)=\mathsf{dom}(\alpha)$ and $\mathsf{codom}(\beta (r,c) \alpha)=\mathsf{codom}(\beta)$. Hence $\beta (r,c) \alpha$ actually consists of three compossible arrows 
\begin{equation*}
r' \xrightarrow{~~\alpha~~} r \xrightarrow{(r,c)} \partial(c)+r\xrightarrow{~~\beta~~} r''.
\end{equation*}

\begin{lem} \label{lem1}
Let $\partial: M \rightarrow N$ be a c-crossed module. Then we have a special congruence $\partial(0)\sim 0$ in $N$.
\end{lem}
	
	\begin{proof}
		We have the special congruence $0+0 \sim 0$ in $M$, from which  the special congruence $\partial(0+0) \sim \partial(0)$ follows, since $\partial$ carries special congruences in $M$ to special congruences  in $N$. From this we have the special congruences $\partial(0)+\partial(0) \sim \partial(0+0) \sim \partial(0)$, which gives the desired special congruence. 
	\end{proof}
	\begin{lem}\label{lem2}
		Let $\partial: M \rightarrow N$ be a cssc-crossed module and $\alpha :r \sim  r' $ a special congruence  in $N$. Then there exists a unique element $c \in M$ with the properties that $c$ is weakly special congruent to $0$ and   $\partial (c)=r'-r$.
	\end{lem}

	\begin{proof}
		By the condition we assume, it follows that there is a special congruence $0\sim r'-r$. From Lemma \ref{lem1} $\partial(0)\sim 0$ is a special congruence, from which we conclude the existence of a special congruence $\partial (0)\sim r'-r$. Since $\partial: M \rightarrow N$  is a special c-crossed module, it follows that there exists a unique element  $c \in M$ with  weakly special congruence $c\sim 0$ and $\partial(c)=r'-r$, which completes the proof.
	\end{proof}
	\begin{cor}\label{cor1}
		If there is a special congruence $\alpha :r \sim  r' $ for $r,r' \in N$, then there exists a unique element $c \in M$ with $\partial(c)=r'-r$, a weak congruence $c\sim  0 $ and an arrow $\varepsilon(r,c): r \rightarrow r'$ in $N \bar{\times} M$, where $\varepsilon: (r'-r)+r \rightarrow r'$ is a special isomorphism corresponding to the special congruence $(r'-r)+r\sim r'$ in $N$ and $ (r,c): r \rightarrow (r'-r)+r$ is an arrow in $N \bar{\times} M$.
	\end{cor}
	\begin{proof}
		Obviously follows from Lemma \ref{lem2}.
	\end{proof}
	
%\subsubsection*{Identification I in $N \bar{\times} M$}
%
%In what follows we will identify  any special isomorphism $\alpha :r \sim  r' $ in $N \bar{\times} M$ with the morphism $\varepsilon(r,c)$  defined uniquely; i.e. $c$ is the element defined uniquely and weakly special congruent to $0$ in $M$ with $\partial (c)=r'-r$ and $\varepsilon$ is a  unique special isomorphism defined in Corollary \ref{cor1}. We will use the notation $\alpha \equiv \varepsilon(r,c)$.
%
%\subsubsection*{Identification II in $N \bar{\times} M$}
%
%If $\varphi:c\sim c'$ is a weak special congruence in $M$, then $\beta (r,c)\alpha\approx\beta' (r,c')\alpha$, for any $r\in N$, where $\beta=\beta'(\partial\varphi+1_r)$, for any $\beta'$ with $\mathsf{dom}(\beta')=\partial(c')+r$, and any $\alpha$ with $\mathsf{codom}(\alpha)=r$
%\begin{equation}\label{diag:31}
%	\begin{gathered}
%		\xymatrix@R=15mm@C=15mm{
%			&  & \partial(c)+r \ar[dr]^-{\beta} \ar[dd]^-{\approx}_-{\partial(\varphi)+1_r} &  \\ 
%			r' \ar[r]^-{\alpha} & r \ar[ur]^-{(r,c)} \ar[dr]_-{(r,c')} &  & r'' \\
%			&  & \partial(c')+r \ar[ur]_-{\beta'} &  
%		}
%	\end{gathered}
%\end{equation}

We now define two identifications in $N \bar{\times} M$ as follows:

\begin{description}
	\item[Identification I in $N \bar{\times} M$:] In what follows we will identify  any special isomorphism $\alpha :r \sim  r' $ in $N \bar{\times} M$ with the arrow $\varepsilon(r,c)$  defined uniquely; i.e. $c$ is the element defined uniquely and weakly special congruent to $0$ in $M$ with $\partial (c)=r'-r$ and $\varepsilon$ is a  unique special isomorphism defined in Corollary \ref{cor1}. We will use the notation $\alpha \equiv \varepsilon(r,c)$.
	\item[Identification II in $N \bar{\times} M$:] If $\varphi:c\sim c'$ is a weak special congruence in $M$, then $\beta (r,c)\alpha\approx\beta' (r,c')\alpha$, for any $r\in N$, where $\beta=\beta'(\partial(\varphi)+1_r)$, for any special congruence $\beta'$ with $\mathsf{dom}(\beta')=\partial(c')+r$, and any $\alpha$ with $\mathsf{codom}(\alpha)=r$
	\begin{equation}\label{diag:31}
		\begin{gathered}
			\xymatrix@R=15mm@C=15mm{
				&  & \partial(c)+r \ar[dr]^-{\beta} \ar[dd]^-{\approx}_-{\partial(\varphi)+1_r} &  \\ 
				r' \ar[r]^-{\alpha} & r \ar[ur]^-{(r,c)} \ar[dr]_-{(r,c')} &  & r'' \\
				&  & \partial(c')+r \ar[ur]_-{\beta'} &  
			}
		\end{gathered}
	\end{equation}
\end{description}

%	\section*{Identification I in $N \bar{\times} M$}
%	In what follows we will identify  any special isomorphism $\alpha :r \sim  r' $ in $N \bar{\times} M$ with the morphism $\varepsilon(r,c)$  defined uniquely; i.e. $c$ is the element defined uniquely and weakly special congruent to $0$ in $M$ with $\partial (c)=r'-r$ and $\varepsilon$ is a  unique special isomorphism defined in Corollary \ref{cor1}. We will use the notation $\alpha \equiv \varepsilon(r,c)$.
%	\section*{Identification II in $N \bar{\times} M$}
%	If $\varphi:c\sim c'$ is a weak special congruence in $M$, then $\beta (r,c)\alpha\approx\beta' (r,c')\alpha$, for any $r\in N$, where $\beta=\beta'(\partial\varphi+1_r)$, for any $\beta'$ with $\mathsf{dom}(\beta')=\partial(c')+r$, and any $\alpha$ with $\mathsf{codom}(\alpha)=r$
%	\begin{equation}\label{diag:31}
%	\begin{gathered}
%	\xymatrix@R=15mm@C=15mm{
%	 &  & \partial(c)+r \ar[dr]^-{\beta} \ar[dd]^-{\approx}_-{\partial(\varphi)+1_r} &  \\ 
%	r' \ar[r]^-{\alpha} & r \ar[ur]^-{(r,c)} \ar[dr]_-{(r,c')} &  & r'' \\
%	 &  & \partial(c')+r \ar[ur]_-{\beta'} &  
%	}
%	\end{gathered}
%	\end{equation}
%	
%	\[
%	\begin{tikzcd}
%		& r \arrow[dd,equal]  \ar{rr}{(r,c)} && \arrow[dd,"\approx " ]  \partial(c)+r \\
%		r' \arrow[ur,"\alpha"] \arrow[dr,"\alpha"] &&
%		&& r'' \arrow[ul,"\beta"]
%		\arrow[dl,"\beta'"]  \\
%		& r\ar{rr}{(r,c')} &&\partial(c')+r 
%	\end{tikzcd}
%	\] 
	In the diagram, in particular,  if   $\beta'=1_{\partial(c')+r}$ then $\beta=\partial(\varphi)+1_r$ and we obtain $(r,c')\alpha=(\partial(\varphi)+1_r)(r,c)\alpha$. If in addition $\alpha=1_r$, then we have $(r,c')=(\partial(\varphi)+1_r)(r,c)$.
	
	It is obvious, but worth to note, that on the base of described identifications we have the following. If $\alpha\equiv\varepsilon(r,c)$ by \textbf{Identification I} and $\varepsilon(r,c)\equiv\varepsilon'(r,c')$ by \textbf{Identification II}, then $\alpha\equiv\varepsilon'(r,c')$; also if $\beta(r,c)\alpha\equiv\beta'(r,c')\alpha\equiv\beta''(r,c'')\alpha$ by \textbf{Identification II}, then $\beta(r,c)\equiv\beta''(r,c'')\alpha$.

	%%Let $(r,c),(r,c') \in N \bar{\times} M$, where $\eta: c \sim c'$ is a  weak special congruence in $M$; i.e. $c \sim c'$ is a congruence and $\partial \eta: \partial(c) \sim \partial(c')$ is a special congruence in $N$. Then we have a special congruence $\bar{\eta} : \partial(c)+r \sim \partial(c')+r$ and the special isomorphism   $\bar{\eta}$ in $N \bar{\times} M$. In what follows we will identify the arrows $\bar{\eta}(r,c)$ and $(r,c')$  in $N \bar{\times} M$ and write $\bar{\eta}(r,c)\equiv (r,c')$.
	
	We denote the resulting set obtained by \textbf{Identifications I} and \textbf{II} by $\mathbb{G}_{1}$; the elements of $\mathbb{G}_{1}$ will be called arrows of the category $\mathbb{G}$.
	
%	(I think it is an identification and not a certain relation, it is transitive, reflexive an symmetric; therefore we do not need to generate equivalence relation by $'' \equiv ''$ identification in $N\bar{\times }M$.) 
	
	\section{Composition of arrows in $\mathbb{G}$}
	
	In this section we define compositions of arrows in $\mathbb{G}$.

\begin{defn} \label{def1} \begin{enumerate}[label=(\textit{\alph{*}}), leftmargin=1cm]
\item\label{def1:a} Let $\alpha$ and $ \beta$  be two composable special isomorphisms in $\mathbb{G}_{1}$. The composition  $ \beta \circ \alpha$ is defined to be $ \beta\alpha$ which is the special isomorphism corresponding to the composition of special congruences  $\alpha$ and $ \beta$ in $N$ which is a special congruence as well.
\item\label{def1:b} If $\beta$ and $ \gamma$  are composable special isomorphisms in $\mathbb{G}_{1}$, the composition $\gamma \circ(\beta (r,c)\alpha)$ is defined as $ (\gamma\beta)(r,c)\alpha$. In analogous way  $(\beta (r,c)\alpha) \circ \delta=\beta (r,c)(\alpha\delta)$, where $\delta$ and $\alpha$ are composable.
From this definition it follows that 
\begin{equation}\label{eq:1} 
	\gamma \circ ((r,c)\alpha)=\gamma (r,c)\alpha
\end{equation} 
\begin{equation}\label{eq:2}
	(\beta (r,c)) \circ \delta=\beta (r,c) \delta
\end{equation}
\item\label{def1:c} Let $\beta (r,c)\alpha$ and $\beta' (r',c')\alpha'\in \mathbb{G}_{1}$. If $\beta$ and $\alpha'$ are composable, we define 
\begin{equation*}
	(\beta' (r',c')\alpha') \circ (\beta (r,c)\alpha)=\theta (r,c'+c)\alpha,
\end{equation*}
where $\theta$ is the composition of the following special isomorphisms
\begin{equation*}
	(\partial(c')+\partial(c))+r\approx \partial(c')+(\partial(c)+r)\overset{{{1}_{\partial(c')}}+\beta }{\mathop{\approx }}\,\partial(c')+{{r}_{2}}\overset{{{1}_{\partial(c')}}+\alpha }{\mathop{\approx }}\,\partial(c')+r'\overset{\beta }{\mathop{\approx }}\,r''
\end{equation*}
the picture is 
\begin{equation*}
	{{r}_{1}}\xrightarrow{\alpha }r\xrightarrow{(r,c)}\partial(c)+r\xrightarrow{\beta }{{r}_{2}}\xrightarrow{\alpha '}r'\xrightarrow{(r',c')}\partial(c')+r'\xrightarrow{\beta '}r''.
\end{equation*}
\end{enumerate}
\end{defn}

Now we shall show that the above definitions of compositions do not depend on the choice of representatives. 

\begin{prop}
In definition \ref{def1:a}, the composition of special isomorphisms does not depend on the choice of representatives.
\end{prop} 

\begin{proof}
Suppose $\alpha \equiv \varepsilon(r,c)$, $\beta \equiv \varepsilon_1 (r',c_1)$, where $\alpha:r \rightarrow r'$, $\beta:r' \rightarrow r_1'$, $\varepsilon: (r'-r)+r \rightarrow r'$ and $\varepsilon_1: (r_1'-r')+r' \rightarrow r_1'$ are special isomorphisms with $\partial(c)= r'-r $, $\partial(c_1)= r_1'-r' $. Let $\beta \alpha=\Bar{\varepsilon} (r,\Bar{c})$ by \textbf{Identification I}, where $\Bar{\varepsilon}: (r_1'-r')+r' \rightarrow r_1'$ is a special isomorphism and $\partial(\Bar{c})=r_1'-r$. We have $\varepsilon_1 (r',c_1)\circ \varepsilon(r,c) =\eta(r,c_1+c)$, where 
\begin{equation*}
\eta: (\partial(c_1)+\partial(c))+r \mathop{\approx }(r_1'-r')+(\partial(c)+r) \mathop{\approx }(r_1'-r')+r'\mathop{\approx }r_1'
\end{equation*}
is the  composition of the special isomorphisms. We have to show the equality $\eta(r,c_1+c)=\Bar{\varepsilon} (r,\Bar{c})$. Since $M$ is connected, we have the congruence $c_1+c\sim \Bar{c} $ in $M$. Moreover $\partial(c_1)+\partial(c)=(r_1'-r')+(r'-r)\mathop{\approx } r_1'-r \mathop{\approx }\partial(c_1)$. By \textbf{Identification II}  and Definition \ref{def1} \ref{def1:b} we obtain the desired equality.
\end{proof}
	\begin{prop}
		Let $\beta (r,c)\alpha \in \mathbb{G}_{1}$,  $\gamma$ and $ \beta$  be composable and  $\gamma \equiv \varepsilon (r',c')$ by \textbf{Identification I}. Then 
		\begin{equation*}
		(\varepsilon (r',c'))\circ (\beta (r,c)\alpha)=(\varepsilon (r',c')\circ \beta)\circ ((r,c)\alpha).
		\end{equation*}
	\end{prop}
	\begin{proof}
		We compute both sides according to the definition of composition \ref{def1:c}. We have 
		\begin{equation*}
		(\varepsilon (r',c'))\circ (\beta (r,c)\alpha)=\theta (r,c'+c)\alpha,
		\end{equation*}
		where
		\begin{equation*}
		\theta:(\partial(c')+\partial(c))+r\mathop{\approx}\partial(c')+(\partial(c)+r)\overset{{{1}_{\partial(c')}}+\beta }{\mathop{\approx }}\partial(c')+r'=(r''-r')+r'\overset{\varepsilon }{\mathop{\approx }} r''
		\end{equation*}
		is the composition of the special identities and the equality.
		
		For the right-side we have 
		\begin{equation} \label{eq:3}
		(\varepsilon (r',c')\circ \beta)\circ ((r,c)\alpha)=(\varepsilon (r',c') \beta))\circ ((r,c)\alpha)=\theta'(r,c'+c),
		\end{equation} 
		where $\theta'=\theta$, which proves the equality. In \eqref{eq:3} we  use \eqref{eq:2}.
	\end{proof}
	\begin{prop}
		Let  $\beta(r,c)\alpha \equiv \beta'(r,c')\alpha$ by \textbf{Identification II}. Then for any $\beta_1(r_1,c_1)\alpha_1 $ with $\mathsf{dom}(\alpha_1) = \mathsf{codom}(\beta)$ there is an equality
		\begin{equation*}
		\beta_1(r_1,c_1)\alpha_1\circ\beta(r,c)\alpha=\beta_1(r_1,c_1)\alpha_1\circ \beta'(r,c')\alpha.	
		\end{equation*}
	\end{prop}
	\begin{proof}
		By the definition of composition the left-side of the equality is equal to $\theta(r,c_1+c)\alpha$ and the right-side is equal to $\theta'(r,c_1+c')\alpha$, where $\theta$ and $\theta'$ are the compositions of the following special isomorphisms 
		\begin{equation*}
		\theta:  (\partial(c_1) +\partial(c))+r \approx \partial(c_1) +(\partial(c)+r)\overset{{{1}_{\partial(c_1)}}+\alpha_1\beta }{\mathop{\approx }} \partial(c_1)+r_1 \overset{\beta_1 }{\mathop{\approx }} r_2	
		\end{equation*}
		\begin{equation*}
		\theta':  (\partial(c_1) +\partial(c'))+r \approx \partial(c_1) +(\partial(c')+r)\overset{{{1}_{\partial(c_1)}}+\alpha_1\beta' }{\mathop{\approx }} \partial(c_1)+r_1 \overset{\beta_1 }{\mathop{\approx }} r_2
		\end{equation*} 
		The picture is 
		
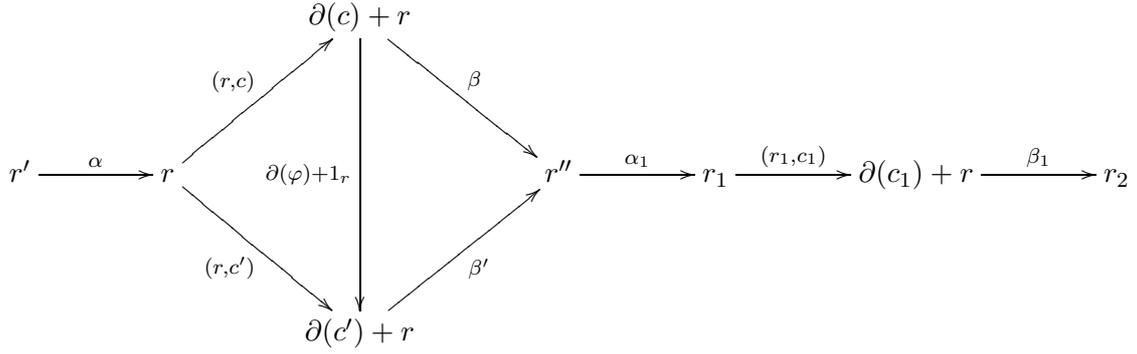
\begin{equation*}
		\xymatrix@R=15mm@C=15mm{
			&  & \partial(c)+r \ar[dr]^-{\beta} \ar[dd]_-{\partial(\varphi)+1_r} &  \\
			r' \ar[r]^-{\alpha} & r \ar[ur]^-{(r,c)} \ar[dr]_-{(r,c')} &  & r'' \ar[r]^-{\alpha_1} & r_1 \ar[r]^-{(r_1,c_1)} & \partial(c_1)+r \ar[r]^-{\beta_1} & r_2\\
			&  & \partial(c')+r \ar[ur]_-{\beta'} &  
		}
		\end{equation*}
%		\[ 
%		\begin{tikzcd}
%			& r  \ar{rr}{(r,c)} && \arrow[dd,"\partial(\varphi)+1_r "] \partial(c)+r\arrow[dr," "] \\
%			r' \arrow[ur,"\alpha"] \arrow[dr,"\alpha"] &&
%			&& r'' \ar{r}{\alpha_1} &r_1 \ar{r}{(r_1,c_1)}&\partial(c_1)+r\ar{r}{\beta_1}&r_2 \\
%			& r\ar{rr}{(r,c')} &&\partial(c')+r \arrow[ur," "]
%		\end{tikzcd}
%		\] 
		where $\varphi:c\sim c'$ is a weak special congruence which is meant in \textbf{Identification II}. Since $c_1+c\approx c_1+c'$ is a weak special congruence by the same \textbf{Identification II} we have the equality $\theta(r,c_1+c)=\theta'(r,c_1+c')$.
		
		The equality when compositions are taken from the right-side can be proved similarly.
	\end{proof}
	\begin{prop}\label{prop:44}
		Suppose $\beta'(r',c')\alpha'=\varepsilon(r',c')$ and $\varepsilon(r',c')\equiv \gamma $ is a special isomorphism in the definition of composition (c). Then $(\varepsilon (r',c'))\circ (\beta (r,c)\alpha)=(\gamma \beta) (r,c)\alpha.$
	\end{prop}
	\begin{proof}
		We have the following notation;
		
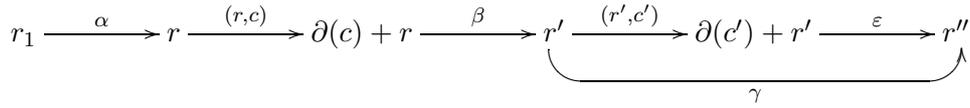
\begin{equation*}
		\xymatrix@R=15mm@C=15mm{
		r_1 \ar[r]^-{\alpha} & r \ar[r]^-{(r,c)} & \partial(c)+r \ar[r]^-{\beta} & r' \ar[r]^-{(r',c')} \ar@<-2pt> `d[r] `[rr]_{\gamma} [rr] & \partial(c')+r' \ar[r]^-{\varepsilon} & r''
		}
		\end{equation*}
%		\[
%		\begin{tikzcd}
%			{{r}_{1}}   \arrow{r}{\alpha } & r
%			\arrow{r}{(r,c)} & \partial(c)+r
%			\arrow{r}{\beta} & r'
%			\arrow{r}{(r',c')} 
%			\arrow[bend right]{rr}{{\gamma}} &
%			\partial(c')+r'  \arrow{r}{\varepsilon}        &
%			r'' 
%		\end{tikzcd}
%		\] 
		
		The left-side of the equality is equal to $\theta (r,c'+c)\alpha$ by the Definition \ref{def1} (c), where $\theta:  \partial(c'+c)+r
		\rightarrow r''$ is an obvious special isomorphism. Here we have to note that $\partial(c')=r''-r'$ and  $\varepsilon: (r''-r')+r'\rightarrow r''$ is a special isomorphism. We have to show that $\theta (r,c'+c)\alpha=(\gamma \beta) (r,c)\alpha$. The equality follows from \textbf{Identification II} coherence and \eqref{eq:2}. The diagram is
		
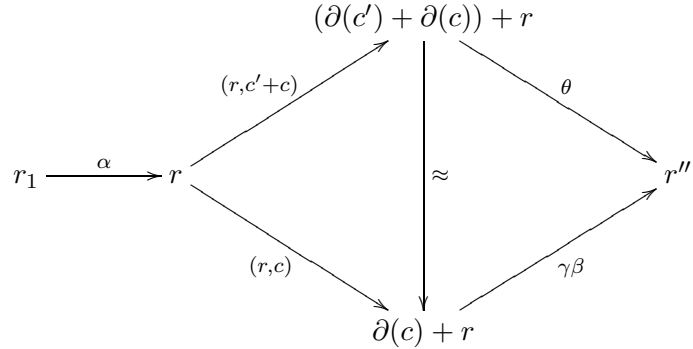
\begin{equation*}
			\xymatrix@R=15mm@C=15mm{
				&  & (\partial(c')+\partial(c))+r \ar[dr]^-{\theta} \ar[dd]^-{\approx} &  \\
				r_1 \ar[r]^-{\alpha} & r \ar[ur]^-{(r,c'+c)} \ar[dr]_-{(r,c)} &  & r'' \\
				&  & \partial(c)+r \ar[ur]_-{\gamma\beta} &  
			}
		\end{equation*} 
%		\[ 
%		\begin{tikzcd}
%			& r \arrow[dd,equal]  \ar{rr}{(r,c'+c)} && \arrow[dd,"\approx "]  (\partial(c')+\partial(c))+r \\
%			r_1 \arrow[ur,"\alpha"] \arrow[dr,"\alpha"] &&
%			&& r'' \arrow[ul,"\theta"]
%			\arrow[dl,"\gamma \beta"]  \\
%			& r\ar{rr}{(r,c)} &&\partial(c)+r 
%		\end{tikzcd}
%		\] 
		Note that the composition 
		\begin{equation*}
		(\partial(c')+\partial(c))+r\approx \partial(c')+(\partial(c)+r)=(r''-r') +\partial(c)+r\approx \partial(c)+r	
		\end{equation*}
		is a special isomorphism since $(r''-r') \approx 0 $ is a special isomorphism.
	\end{proof}

	In what follows we will not use the sign ``$\circ$'' for composition. 
	
	\begin{prop}\label{prop:4.6}
		For any three composible arrows we have an equality
		\begin{equation*}
		(r,c)((r_2,c_2)(r_1,c_1))=((r,c)(r_2,c_2))(r_1,c_1).	
		\end{equation*}
	\end{prop}
	
	\begin{proof}
		We compute both sides of the equality and prove that they are equal. The right-side is equal to $\psi_1 \psi (r_1,(c+c_2)+c_1)$, where $\psi$ and $\psi_1$ are the special isomorphisms
		\begin{equation*}
		((\partial(c)+\partial(c_2))+\partial(c_1))+r_1 \xrightarrow{\psi} (\partial(c)+\partial(c_2))+(\partial(c_1)+r_1)\xrightarrow{\psi_1} \partial(c)+r.	
		\end{equation*}
		Here we have in mind that since the three arrows are composible we have $\partial(c_1)+r_1=r_2$ and $\partial(c_2)+r_2=r$.
		
		The left-side of the equality is equal to $\varphi_1 \varphi (r_1,c+(c_2+c_1))$, where $\varphi$ and $\varphi_1$ are the special isomorphisms 
		\begin{equation*}
		(\partial(c)+(\partial(c_2)+\partial(c_1)))+r_1 \xrightarrow{\varphi} (\partial(c)+\partial(c_2))+(\partial(c_1)+r_1)\xrightarrow{\varphi_1} \partial(c)+r.	
		\end{equation*}
		The diagram is 
		\begin{equation*}
		\xymatrix@R=15mm@C=15mm{
		  & ((\partial(c)+\partial(c_2))+\partial(c_1))+r \ar[dr]^-{\psi_1\psi} \ar[dd]^-{\partial(\eta)+1_r} &  \\
		 r \ar[ur]^-{(r,(c+c_2)+c_1)~~} \ar[dr]_-{(r,c+(c_2+c_1))~~} &  & \partial(c)+r \\
		& (\partial(c)+(\partial(c_2)+\partial(c_1)))+r \ar[ur]_-{\varphi_1 \varphi} &  
		}
		\end{equation*}
%		\[
%		\begin{tikzpicture}
%			\node (s) {$r$};
%			\node (xy) [below=2 of s] {$(\partial(c)+(\partial(c_2)+\partial(c_1)))+r $};
%			\node (x) [right=3 of xy] {$\partial(c)+r$};
%			\node (y) [right=3 of s] {$((\partial(c)+\partial(c_2))+\partial(c_1))+r$};
%			\draw[->] (s) to node [sloped, above] {$(r,(c+c_2)+c_1)$} (y);
%			\draw[->] (xy) to node [sloped, above] {$\gamma\eta +!_r$} (y);
%			\draw[->] (s) to node [anchor=east]{$(r,c+(c_2+c_1))$} (xy);
%			\draw[->] (xy) to node [below]{$\varphi_1 \varphi$} (x);
%			%\draw[->] (s) to node [below] {$(p_2, \Pi_2)$} (y);
%			\draw[->] (y) to node [below][anchor=west] {$\psi_1 \psi$} (x);
%		\end{tikzpicture}
%		\]
		where $\eta: (c+c_2)+c_1 \sim c+(c_2+c_1) $ is a special congruence of associativity in $M$ and $\partial \eta$ is a special isomorphism in $\mathbb{G}_{1}$. From the \textbf{Identification II} of arrows we have commutative triangle on the left-side of the diagram. The second triangle is also commutative from the coherence property of special congruences in c-groups. Applying Definition \ref{def1} (b), this commutative diagram gives the proof.
	\end{proof}
	
	Note that associativity of composition in the general case is proved in a similar way. Here we mean that the arrows are of the type $\beta (r,c)\alpha$.
	
	\begin{prop}\label{prop:46}
		Let $(r,c),(r',c') \in \mathbb{G}_{1}$ and $r \sim r'$ be a congruence in $N$. Then there exist arrows $\varphi: r \rightarrow r'$ and $\theta:\partial(c)+r \rightarrow \partial(c')+r'$ for which the following diagram commutes
		\begin{equation*}
		\xymatrix@R=15mm@C=15mm{
		r \ar[r]^-{(r,c)} \ar[d]_-{\varphi} & \partial(c)+r \ar[d]^-{\theta} \\
		r' \ar[r]_-{(r',c')} & \partial(c')+r'
		}
		\end{equation*}
%		\[
%		\begin{tikzpicture}
%			\node (s) {$r$};
%			\node (xy) [below=2 of s] {$r' $};
%			\node (x) [right=3 of xy] {$\partial(c')+r'$};
%			\node (y) [right=3 of s] {$(\partial(c)+r)$};
%			\draw[->] (s) to node [sloped, above] {$(r,c)$} (y);
%			\draw[->] (s) to node [anchor=east]{$\psi$} (xy);
%			\draw[->] (xy) to node [below]{$(r',c')$} (x);
%			\draw[->] (y) to node [below][anchor=west] {$\theta$} (x);
%		\end{tikzpicture}
%		\]
	\end{prop}
	\begin{proof}
		Since c-crossed module $\partial: M \rightarrow N$ is connected, we have the congruence $c \sim c'$. Since it is a special c-crossed module there exists (not necessarily unique) an element $c''$ with 
		\begin{equation}\label{eq:parc2}
		\partial(c'')=r'-r
		\end{equation}
		and $c'' \sim 0$. Let $\varepsilon:(r'-r)+r \rightarrow r'$ be the special isomorphism. Then $\varepsilon (r,c'')$ is an arrow $r \rightarrow r'$. In a similar way, we have a congruence $\partial(c) +r \sim \partial(c')+r'$ and therefore there exists an element $c_1$ with  
		\begin{equation}\label{eq:parc1}
		\partial(c_1)=(\partial(c')+r')-(\partial(c) +r),
		\end{equation}
	$c_1 \sim 0$ and we have a special isomorphism $\varphi:((\partial(c')+r')-(\partial(c) +r))+(\partial(c) +r) \rightarrow \partial(c')+r'$. Therefore there is an arrow   $\varphi(\partial(c) +r,c_1):\partial(c) +r \rightarrow \partial(c')+r'$. We take $\psi=\varepsilon (r,c'')$ and $\theta=\varphi(\partial(c) +r,c_1) $; and prove that the diagram is commutative. The compositions in the diagram are respectively equal to $\chi_1 \chi(r,c'+c'')$ and $\sigma_1 \sigma (r,c_1+c)$ where $\chi_1 \chi$, $\sigma_1 \sigma$ are the following special isomorphisms 
	\begin{equation*}
	(\partial(c')+\partial(c''))+r \xrightarrow{\chi} \partial(c')+(\partial(c'')+r)\xrightarrow{\chi_1} \partial(c')+r',	
	\end{equation*}
	\begin{equation*}
	(\partial(c_1)+\partial(c))+r \xrightarrow{\sigma } \partial(c_1)+(\partial(c)+r)\xrightarrow{\sigma_1} \partial(c')+r'.	
	\end{equation*}
	 Here we applied \eqref{eq:parc2} and \eqref{eq:parc1}. Since $M$ is connected we have a congruence $c'+c''\sim c_1+c$. Moreover $\partial(c')+\partial(c'')=\partial(c')+(r'-r)$ and $\partial(c_1)+\partial(c) $ is special congruent to $\partial(c')+(r'-r)$; and therefore  $c'+c''$ is weakly special congruent to $c_1+c$. On the base of \textbf{Identification II} the diagram commutes.
	\end{proof}
	\begin{cor}\label{cor:2}
		Let $\beta (r,c)\alpha$ and $ \beta' (r',c')\alpha' \in \mathbb{G}_{1}$, where $\mathsf{dom}(\alpha)=\mathsf{dom}(\alpha')$ and $\mathsf{codom}(\beta)=\mathsf{codom}(\beta')$. If $c\sim c'$ is a weak special congruence, then $\beta (r,c)\alpha=\beta' (r',c')\alpha'$.

	\end{cor}
	\begin{proof}
		The proof follows from Proposition \ref{prop:46} and the coherence property of special isomorphisms in $\mathbb{G}_{1}$ as in the diagram 
		
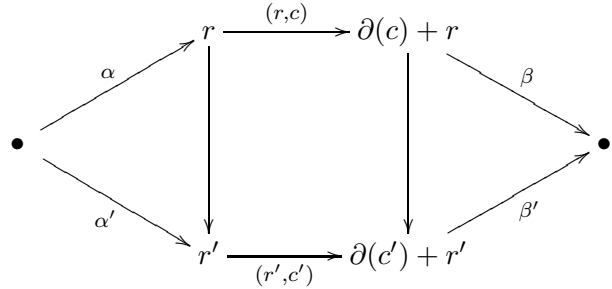
\begin{equation*}
		\xymatrix@R=10mm@C=15mm{
			& r \ar[r]^-{(r,c)} \ar[dd] & \partial(c)+r \ar[dr]^-{\beta} \ar[dd] &  \\
			~~~~~\bullet~~~~~ \ar[ur]^-{\alpha} \ar[dr]_-{\alpha'} &  &  &  \bullet \\
			& r' \ar[r]_-{(r',c')} & \partial(c')+r' \ar[ur]_-{\beta'} &  
		}
		\end{equation*}
%		\[
%		\begin{tikzpicture}
%			\node (s) {$\bullet$};
%			\node (xy)[below right= of s] {$r'$ };
%			\node (xyz) [above right= of s] {$r$ };
%			\node (z) [ right=3 of xyz] {$\partial(c)+r$ };
%			\node (x) [right=3 of xy] {$\partial(c')+r' $};
%			\node (y) [below right = of z]{$\bullet $};
%			
%			\draw[->] (s) to node [anchor=east] {$\alpha'$} (xy);
%			\draw[->] (s) to node [anchor=east]{$\alpha$} (xyz);
%			\draw[->] (xyz) to node [sloped, above]{$(r,c)$} (z);
%			\draw[->] (xyz) to node [] {} (xy);
%			\draw[->] (z) to node []{} (x);
%			%\draw[->] (z) to node [anchor=west] {$\varphi^{-1}$} (y);
%			%\draw[->] (x) to node [anchor=west] {$\varphi^{-1}$} (y);
%			\draw[->] (xy) to node [sloped, below] {$(r',c')$} (x);
%			
%			\draw[->] (z) to node [anchor=west]{$\beta$} (y);
%			\draw[->] (x) to node [anchor=west]{$\beta'$} (y);
%		\end{tikzpicture}
%		\]
		
\end{proof}
\section{Identity arrows and invertibility of arrows in $\mathbb{G}_1$ }
	
For each object $r \in\mathbb{G}_0$, we define an arrow $\varepsilon (r,0)$, where $(r,0): r \rightarrow \partial(0)+r$ is an arrow in $\mathbb{G}_1$ and $\varepsilon:\partial(0)+r \rightarrow r $ is a special isomorphism by means of the special isomorphism $\partial(0)\approx 0$. We prove that $\varepsilon (r,0):r \rightarrow r $ is an identity arrow.
\begin{lem}
For any $ (r_1,c)$ with $\partial(c)+r_1=r$ and any $(r,c_1)\in \mathbb{G}_1$ there are equalities
\begin{equation*}
	\varepsilon (r,0)(r_1,c)=(r_1,c),	
\end{equation*}
\begin{equation*}
	(r,c_1)\varepsilon (r,0)=(r,c_1).	
\end{equation*}
\end{lem}
\begin{proof}
From the definition of composition in $\mathbb{G}_1$ we have $\varepsilon (r,0)(r_1,c)=\varepsilon \varepsilon_1 (r_1,0+c)$, where $ \varepsilon_1:(\partial(0)+\partial(c))+r_1 \rightarrow \partial(0)+r $ is a special isomorphism by $\partial(c)+r_1=r$. Moreover, we have the special congruence $0+c\sim c$. By the Identification II we obtain commutative diagram 
		\begin{equation*}
			\xymatrix@R=15mm@C=15mm{
				r_1 \ar[r]^-{(r_1,c)} \ar[d]_-{(r_1,0+c)} & \partial(c)+r_1=r  \\
				(\partial(0)+\partial(c))+r_1 \ar[r]_-{\varepsilon_1} & \partial(0)+r \ar[u]_-{\varepsilon}
			}
		\end{equation*}
%		\[
%		\begin{tikzpicture}
%			\node (s) {$r_1$};
%			\node (xy) [below right= of s] {$(\partial(0)+\partial(c))+r_1  $};
%			\node (x) [right=6 of s] {$\partial(c)+r_1=r$};
%			\node (y) [below right= of xy] {$\partial(0)+r  $};
%			\draw[->] (s) to node [sloped, above] {$(r_1,c)$} (x);
%			\draw[->] (s) to node [anchor=east]{$(r_1,0+c)$} (xy);
%			\draw[->] (xy) to node [below]{$\varepsilon_1$} (y);
%			\draw[->] (y) to node [below][anchor=west] {$\varepsilon$} (x); 
%		\end{tikzpicture}
%		\]
		which proves the first equality. The second equality is proved in a similar way.
\end{proof}
	
Hence we have a map $i:\mathbb{G}_0\rightarrow \mathbb{G}_1 $ defined by $i(r)=\varepsilon(r,0)$, which is the identity arrow $1_r$ for any object $r$ in $\mathbb{G}_0$.
		
For any arrow $\beta (r,c)\alpha \in \mathbb{G}_1$ define an arrow $\alpha^{-1} (r,c)^{-1}\beta^{-1}$, where $(r,c)^{-1}=\varphi(\partial(c)+r,-c)$, and $\varphi:\partial(-c)+(\partial(c)+r)\rightarrow r$ is a special isomorphism.
		
The pictures are as follows 
\begin{equation}\label{eq:4} 
\xymatrix@C=15mm{r' \ar[r]^-{\alpha} & r\ar[r]^-{(r,c)} & \partial(c)+r \ar[r]^-{\beta} & r'' }	
\end{equation} 
\begin{equation}\label{eq:5}
\xymatrix@C=17mm{
r'' \ar[r]^-{\beta^{-1}} & \partial(c)+r \ar[r]^-{(\partial(c)+r,-c)} & \partial(-c)+(\partial(c)+r) \ar[r]^-{\varphi} & r \ar[r]^-{\alpha^{-1}} & r' 
}
\end{equation}
		
\begin{prop}
For any $\beta (r,c)\alpha \in \mathbb{G}_1$ we have 
\begin{equation*}
(\beta (r,c)\alpha)(\alpha^{-1}\varphi(\partial(c)+r,-c)\beta^{-1})=1_{r''},
\end{equation*}
\begin{equation*}
(\alpha^{-1}\varphi(\partial(c)+r,-c)\beta^{-1})(\beta (r,c)\alpha)=1_{r'}.	
\end{equation*}
\end{prop}
	
\begin{proof}
By the definition of composition, we obtain that the left-side of the first equality is equal to 
\begin{equation*}
\beta\xi(\partial(c)+r,c-c)\beta^{-1}	
\end{equation*}
where $\xi:\partial(c-c)+ (\partial(c)+r)\rightarrow\partial(c)+r$ is a special isomorphism and $\beta$ is as given in \eqref{eq:4}. We have the following diagram 
		
%\[\begin{tikzpicture}
%			\node (s) {$r''$};
%			\node (xy) [ right=2 of s] {$\partial(c)+r$};
%			\node (x) [right=3 of xy] {$\partial(c-c)+(\partial(c)+r)$};
%			\node (y) [below =4 of xy] {$\partial(0)+(\partial(c)+r)$};
%			\node (z) [below =3 of x] {$\partial(c)+r  $};
%			\node (w) [below =3 of y] {$\partial(c)+r  $};
%		\node (k) [below =4 of z] {$r''  $};
%			\draw[->] (s) to node [sloped, above] {$\beta^{-1}$} (xy);
%			\draw[->] (xy) to node [sloped, above] {$(\partial(c)+r,c-c)$} (x);
%			\draw[->] (xy) to node [anchor=west] {$(\partial(c)+r,0)$} (y);
%			\draw[->] (y) to node  [anchor=east] {$\varepsilon$} (w);
%			\draw[->] (x) to node [anchor=west] { $\xi$} (z);
%			\draw[->] (x) to node [sloped,above] { $\varphi+1_{\partial(c)+r}$} (y);
%			\draw[->] (z) to node [anchor=west]{$\beta$} (k);
%			\draw[->] (z) to node [sloped,above]{$=$} (w);
%			\draw[->] (w) to node [below] {$\beta$} (k);
%			\draw[->,bend right] (xy) to node [anchor=east] {$1_{\partial(c)+r}$} (w); 
%		\end{tikzpicture}\]
\begin{equation*} %R/C=38/13
	\xymatrix@R=19mm@C=15mm{
	r'' \ar[r]^-{\beta^{-1}} & \partial(c)+r \ar[r]^-{(\partial(c)+r,c-c)} \ar[d]^-{(\partial(c)+r,0)} %\ar@<-2pt> `l[ld] `[dd]^{1_{\partial(c)+r}} [dd] 
	\ar@/_4pc/[dd]_-{1_{\partial(c)+r}}  & \partial(c-c)+(\partial(c)+r) \ar[d]^-{\xi} \ar[dl]^-{\varphi+1_{\partial(c)+r}} \\
	& \partial(0)+(\partial(c)+r) \ar[d]_-{\varepsilon} & \partial(c)+r \ar[d]^-{\beta} \ar@{=}[dl]^-{1_{\partial(c)+r}} \\
	& \partial(c)+r \ar[r]_-{\beta} & r''
	}
\end{equation*}
For the arrows $(\partial(c)+r,c-c)$ and $(\partial(c)+r,0)$ we have the special congruence $c-c\sim 0$ and therefore a special isomorphism $\varphi:\partial(c-c)\rightarrow \partial (0)$. From  \textbf{Identification II}, it follows that the upper triangular diagram commutes. In the diagram, $\varepsilon$ is a special isomorphism and the left vertical composition is $1_{\partial(c)+r}$. From the coherence property of special congruences in c-groups, and therefore from the coherence of special isomorphisms in $\mathbb{G}_1$, we conclude that the diagram is commutative, which proves the first equality of the proposition.
		
In a similar way the second equality can be proved.
\end{proof}

We now prove that the inverse of an arrow does not depend on the choise of representatives.

\begin{prop}
If $\alpha: r\rightarrow r'$ is a special isomorphism and $\alpha \equiv\varepsilon(r,c')$ by \textbf{Identification I}, where $\partial(c')=r'-r$ and $\varepsilon$ is a special isomorphism $(r'-r)+r\rightarrow r'$, then $\alpha^{-1}=(\varepsilon(r,c'))^{-1}$.
\end{prop}

	\begin{proof}
		By the definition of inverse arrow $(\varepsilon(r,c'))^{-1}=\varphi(\partial(c')+r,-c')\varepsilon^{-1}$, where $\varphi:\partial(-c')+(\partial(c')+r) \rightarrow r$ is a special isomorphism. Since $\alpha^{-1}$ is a special isomorphism, by \textbf{Identification I} we have $\alpha^{-1}=(\bar{\varepsilon}(r',\Bar{c_1}))$, where $\partial\Bar{c_1}=r-r'$ and $\bar{\varepsilon}:\partial\Bar{c_1}+r'\rightarrow r $ is a special isomorphism. We have to show that $\varphi(\partial(c')+r,-c')\varepsilon^{-1}=\bar{\varepsilon}(r',\Bar{c_1}')$. We have the special isomorphism $\varepsilon^{-1}:r' \rightarrow \partial(c')+r $; and moreover $-c'\sim\Bar{c_1}'$, since $M$ is connected. $\partial( -c')$ is a special congruent to $\Bar{c_1}'$, since both are special congruent to $r-r'$. Like in Proposition \ref{prop:46} the following diagram commutes
		\begin{equation*}
			\xymatrix@R=10mm@C=20mm{
				r' \ar[r]^-{(r',\overline{c}_{1}')} \ar[dd]_-{\varepsilon^{-1}} & \partial(\overline{c}_1')+r' \ar[d]^-{\overline{\varepsilon}} \\
				 & r \ar[d]^-{\varphi^{-1}} \\
				\partial(c')+r \ar[r]_-{(\partial(c')+r,-c')} & \partial(-c')+(\partial(c')+r) 
			}
		\end{equation*}
%		\[
%		\begin{tikzpicture}
%			\node (s) {$r'$};
%			\node (xy) [below=3 of s] {$\partial(c')+r$ };
%			\node (x) [right=3 of xy] {$\partial (-c')+(\partial(c')+r)$};
%			\node (y) [right=4 of s] {$(\partial \Bar{c_1}'+r')$};
%			\node (z) [below= of y] {$r$};
%			\draw[->] (s) to node [sloped, above] {$(r',\Bar{c_1}')$} (y);
%			\draw[->] (s) to node [anchor=east]{$\varepsilon^{-1}$} (xy);
%			\draw[->] (y) to node [anchor=west]{$\bar{\varepsilon}$} (z);
%			\draw[->] (xy) to node [sloped, below] {$(\partial(c')+r,-c')$} (x);
%			\draw[->] (z) to node [anchor=west] {$\varphi^{-1}$} (x);
%		\end{tikzpicture}
%		\]
		which proves the equality.
	\end{proof}
	
\begin{prop}
If $\varphi:c\sim c'$ is a weak special congruence, then according to diagram \eqref{diag:31} for any $r \in \mathbb{G}_0$ we have 
\begin{equation*}
(\beta' (r,c')\alpha)^{-1}=(\beta (r,c)\alpha)^{-1}.
\end{equation*}
\end{prop}
\begin{proof}
For the right-side of the equality we have 
\begin{equation*}
(\beta (r,c)\alpha)^{-1}=\alpha^{-1}(\eta(\partial(c)+r,-c)\beta^{-1})	
\end{equation*}
where $\eta:\partial(-c)+(\partial(c)+r)\rightarrow r$ is a special isomorphism. For the left-side of the equality we have
\begin{equation*}
(\beta' (r,c')\alpha)^{-1}=\alpha^{-1}(\eta'(\partial(c')+r,-c')\beta'^{-1})	
\end{equation*}
where $\eta':\partial(-c')+(\partial(c')+r)\rightarrow r$ is a special isomorphism. By diagram \eqref{diag:31} we have the following picture
\begin{equation*}
\xymatrix@R=15mm@C=25mm{
	& \partial(c)+r \ar[r]^-{(\partial(c) +r,-c)} \ar[dd] & \partial(-c)+(\partial(c) +r)\ar[dr]^-{\eta} \ar[dd] &  &  \\
	r'' \ar[ur]^-{\beta^{-1}} \ar[dr]_-{{\beta'}^{-1}} &  &  & r \ar[r]^-{\alpha^{-1}}& r' \\
	& \partial(c')+r \ar[r]_-{(\partial(c')+r,-c')} & \partial(-c')+(\partial(c')+r) \ar[ur]_-{\eta'} &  & 
	}
\end{equation*}
%		\[
%		\begin{tikzpicture}
%			\node (s) {$r''$};
%			\node (xy)[below right=2 of s] {$\partial(c')+r$ };
%			\node (xyz) [above right=2 of s] {$\partial(c)+r$ };
%			\node (z) [ right=4 of xyz] {$(\partial (-c)+(\partial(c)+r)$ };
%			\node (x) [right=4 of xy] {$\partial (-c')+(\partial(c')+r)$};
%			\node (m) [ right= of z] {$r$};
%			\node (n) [right= of x] {$r$};
%			\node (y) [below right =2 of m]{$r'$};
%			
%			\draw[->] (s) to node [anchor=east] {$\beta'^{-1}$} (xy);
%			\draw[->] (s) to node [anchor=east]{$\beta^{-1}$} (xyz);
%			\draw[->] (xyz) to node [sloped, above]{$(\partial(c)+r,-c)$} (z);
%			\draw[->] (xyz) to node [midway,left] {$\partial(\varphi) +1_r$} (xy);
%			\draw[->] (z) to node [anchor=east]{$\partial(-\varphi)+(\partial(\varphi) +1_r)$} (x);
%			%\draw[->] (z) to node [anchor=west] {$\varphi^{-1}$} (y);
%			%\draw[->] (x) to node [anchor=west] {$\varphi^{-1}$} (y);
%			\draw[->] (xy) to node [sloped, below] {$(\partial(c')+r,-c')$} (x);
%			\draw[->] (z) to node [sloped, below]{$\eta$} (m);
%			\draw[->] (x) to node [sloped, below]{$\eta'$} (n);
%			\draw[->] (m) to node [anchor=west]{$\alpha^{-1}$} (y);
%			\draw[->] (n) to node [anchor=west]{$\alpha^{-1}$} (y);
%		\end{tikzpicture}
%		\]
		
By the definition of composition (b) we have
\begin{equation*}
\alpha^{-1}(\eta(\partial(c)+r,-c)\beta^{-1})=(\alpha^{-1}\eta)(\partial(c)+r,-c)\beta^{-1},	
\end{equation*}
\begin{equation*}
\alpha^{-1}(\eta'(\partial(c')+r,-c')\beta'^{-1})=(\alpha^{-1}\eta')(\partial(c')+r,-c')\beta'^{-1}.	
\end{equation*}
Since $\partial(c)+r\approx \partial(c')+r$ is a special isomorphism and $-c\sim-c'$ is a weak special congruence, by Corollary \ref{cor:2} we obtain the equality. 
\end{proof}
	
\section{Addition operation, the zero element and the opposite arrows in $\mathbb{G}_1$}
	
For any arrows $\beta_1(r_1,c_1)\alpha_1$ and $ \beta_2(r_2,c_2)\alpha_2\in\mathbb{G}_1$ their sum is defined by
\begin{equation*}
\beta_1(r_1,c_1)\alpha_1+\beta_2(r_2,c_2)\alpha_2=(\beta_1+\beta_2)\theta(r_1+r_2,c_1+r_1\cdot c_2)(\alpha_1+\alpha_2),
\end{equation*}
where 
\begin{equation*}
\theta:\partial(c_1)+(r_1+(\partial(c_2)-r_1))+r_1+r_2\longrightarrow(\partial(c_1)+r_1)+(\partial(c_2)+r_2)	
\end{equation*}
is the special isomorphism. That can be pictured as 
\begin{equation*}
	\xymatrix@C=25mm{r_1' \ar[r]^{\alpha_1} & r_1 \ar[r]^-{(r_1,c_1)} & \partial(c_1)+r_1 \ar[r]^-{\beta_1} & r_1''}
\end{equation*}
\begin{equation*}
	\xymatrix@C=25mm{r_2' \ar[r]^{\alpha_2} & r_2 \ar[r]^-{(r_2,c_2)} & \partial(c_2)+r_2 \ar[r]^-{\beta_2} & r_2''}
\end{equation*}
\begin{equation*}
	\xymatrix@C=30mm@R=25mm{r_1'+r_2' \ar[d]_-{\alpha_1+\alpha_2} \ar[rr]^{\beta_1(r_1,c_1)\alpha_1+\beta_2(r_2,c_2)\alpha_2} & & r_1''+r_2''\\
	r_1+r_2 \ar[r]_-{(r_1+r_2,c_1+r_1\cdot c_2)} & \partial(c_1)+(r_1+(\partial(c_2)-r_1))+r_1+r_2 \ar[r]_-{\theta} & (\partial(c_1)+r_1)+(\partial(c_2)+r_2) \ar[u]_-{\beta_1+\beta_2} }
\end{equation*}

Let $\alpha\colon r\rightarrow r'$ be a special isomorphism in $\mathbb{G}_1$. By \textbf{Identification I}, we have $\alpha\equiv\varepsilon(r,c)$, where $c\sim 0$ is a weak special congruence with $\partial(c)=r'-r$ and $\varepsilon\colon (r'-r)+r\rightarrow r'$ is the special isomorphism. For any $\gamma(r',c')\beta\in\mathbb{G}_1$ the sum $\alpha+\gamma(r',c')\beta$  is defined to be $\varepsilon(r,c)+\gamma(r',c')\beta$. Similarly, $\gamma(r',c')\beta+\alpha$ is defined to be $\gamma(r',c')\beta+\varepsilon(r,c)$.
	
Now we prove that the addition operation in $\mathbb{G}_1$ does not depend on the choice of representatives. In the case of \textbf{Identification I} of arrows we do not need to prove anything for the definitions of the sums $\alpha+\gamma(r',c')\beta$ and $\gamma(r',c')\beta+\alpha$.
	
\begin{prop}
Let $\alpha\colon r\rightarrow r'$ and $\beta\colon r_1\rightarrow r_1'$ be the special isomorphisms. Then we have the special isomorphism $\alpha+\beta\colon r+r_1\rightarrow r'+r_1'$, which corresponds to the special congruence $\alpha+\beta$ in $M$. We have the identifications $\alpha\equiv\varepsilon(r,c)$, $\beta\equiv\varepsilon_1(r_1,c_1)$ and $\alpha+\beta\equiv\overline{\varepsilon}(r+r_1,\overline{c})$, where $\partial(c)=r'-r$, $\partial(c_1)=r_1'-r_1$, $\partial(\overline{c})=(r'+r_1')-(r+r_1)$, $\varepsilon\colon(r'-r)+r\rightarrow r'$, $\varepsilon_1\colon (r'-r_1)+r_1\rightarrow r_1'$ and $\overline{\varepsilon}\colon((r'+r_1')-(r+r_1))+(r+r_1)\rightarrow r'+r_1'$ are the special isomorphisms. Moreover $c,c_1$  and $\overline{c}$ are weakly special congruent to $0$. Then we have the equality 
\begin{equation*}
\overline{\varepsilon}(r+r_1,\overline{c})=\varepsilon(r,c)+\varepsilon_1(r_1,c_1).
\end{equation*}
\end{prop}
\begin{proof}
		According to the definition of sum in $\mathbb{G}_1$ we have 
		\begin{equation*}
			\varepsilon(r,c)+\varepsilon_1(r_1,c_1)=(\varepsilon+\varepsilon_1)\eta(r+r_1,c+r\cdot c_1)
		\end{equation*}
		where 
		\begin{multline}\label{eq:eta}
			\eta\colon \partial(c+r\cdot c_1)+(r+r_1)=\partial(c)+(r+(\partial(c_1)-r))+(r+r_1)\\ \stackrel{=}{\rightarrow} ((r'-r)+(r+((r_1'-r_1)-r)))+(r+r_1)\rightarrow ((r'-r)+r)+((r_1'-r_1)+r_1)
		\end{multline}
		is a special isomorphism. For two arrows $(r+r_1, c+r\cdot c_1)$ and $(r+r_1,\overline{c})$, we have that $c+r\cdot c_1 \sim c'$, and from  \eqref{eq:eta} $\partial(c+r\cdot c_1)=(r'-r)+((r+(r_1'-r_1))-r)$; and the right-side of this equality is special congruent to $(r'+r_1')-(r+r_1)$, which is equal to $\partial(\overline{c})$. Therefore we can apply \textbf{Identification II}, which gives a commutative diagram 
		\begin{equation*}
			\xymatrix@C=20mm@R=20mm{
				r+r_1 \ar[dd]_-{(r+r_1,\overline{c})} \ar[r]^{(r+r_1,c+r\cdot c_1)} & \partial(c+r\cdot c_1)+(r+r_1) \ar[d]^{\eta} \\
				& ((r'-r)+r)+((r_1'-r_1)+r_1) \ar[d]^{\varepsilon+\varepsilon_1} \\
				((r'+r_1')-(r+r_1))+(r+r_1) \ar[r]_-{\overline{\varepsilon}} & r'+r_1'
			} 
		\end{equation*}
		from which the desired equality follows.
	\end{proof}

	\begin{prop}
		If $\phi\colon c\sim c'$ is a weak special congruence, then for any $r \in \mathbb{G}_0 $ and \textbf{Identification II} $\beta'(r,c')\alpha\equiv\beta(r,c)\alpha$  we have the following equalities 
		\begin{equation*}
			\beta'(r,c')\alpha+\beta_1(r_1,c_1)\alpha_1=\beta(r,c)\alpha+\beta_1(r_1,c_1)\alpha_1
		\end{equation*}
		\begin{equation*}
			\beta_1(r_1,c_1)\alpha_1+\beta'(r,c')\alpha=\beta_1(r_1,c_1)\alpha_1+\beta(r,c)\alpha
		\end{equation*}
		for any arrow $\beta_1(r_1,c_1)\alpha_1$ in $\mathbb{G}_1$.
	\end{prop}
	\begin{proof}
		The right-side of the first equality is equal to $(\beta+\beta_1)\theta(r+r_1,c+r\cdot c_1)(\alpha+\alpha_1)$ and the left-side of the same equality is equal to $(\beta'+\beta'_1)\theta'(r+r_1,c'+r\cdot c_1)(\alpha+\alpha_1)$. We have the following picture (see also diagram \eqref{diag:31}).
  
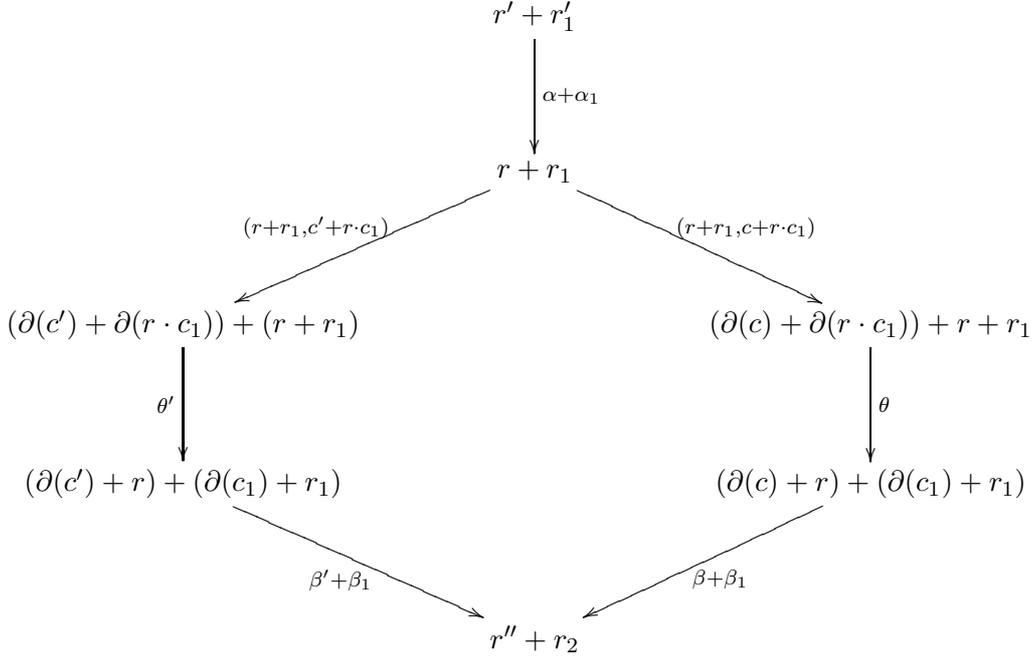
		\begin{equation*}
			\xymatrix@C=15mm@R=15mm{
				& r'+r_1' \ar[d]^-{\alpha+\alpha_1} & \\
				& r+r_1 \ar[dl]_-{(r+r_1,c'+r\cdot c_1)~~} \ar[dr]^-{~~(r+r_1,c+r\cdot c_1)} & \\
				(\partial(c')+\partial(r\cdot c_1))+(r+r_1) \ar[d]_-{\theta'}&  & (\partial(c)+\partial(r\cdot c_1))+r+r_1 \ar[d]^-{\theta} \\
				(\partial(c')+r)+(\partial(c_1)+r_1) \ar[dr]_-{\beta'+\beta_1} & & (\partial(c)+r)+(\partial(c_1)+r_1) \ar[dl]^-{\beta+\beta_1} \\
				& r''+r_2 &
			}
		\end{equation*}
		Since we have a weak special congruence $c'+r\cdot c_1 \sim c+r\cdot c_1$, by \textbf{Identification II} we obtain the first equality of the proposition. The second equality is proved in a similar way.
	\end{proof}
	
	\begin{prop}\label{prop:assoc}
		For any three arrows $(r,c)$, $(r_1,c_1)$ and $(r_2,c_2)$ in $\mathbb{G}_1$ we have an isomorphism
		\begin{equation*}
			((r,c)+(r_1,c_1))+(r_2,c_2)\approx(r,c)+((r_1,c_1)+(r_2,c_2)).
		\end{equation*}
	\end{prop}
	\begin{proof}
		First we compute the left-side of the isomorphism. Let 
		\begin{equation*}
		\theta_1\colon (\partial(c)+(r+(\partial(c_1)-r)))+(r+r_1)\rightarrow (\partial(c)+r)+(\partial(c_1)+r_1)	
		\end{equation*}
		be the special isomorphism. We have
		\begin{equation*}
			\theta_1(r+r_1,c+r\cdot c_1)+(r_2,c_2)=\theta_2((r+r_1)+r_2,(c+r\cdot c_1)+(c+r\cdot c_1)\cdot c_2),
		\end{equation*}
		where 
		\begin{equation*}
			\theta_2\colon ((\partial(c)+(r+(\partial(c_1)-r)))+((r+r_1)+(\partial(c_2)-(r+r_1))))+((r+r_1)+r_2) \rightarrow ((\partial(c)+r)+(\partial(c_1)+r_1))+(\partial(c_2)+r_2)
		\end{equation*}
		is the obvious special isomorphism.
		
		The right-side of the isomorphism in the proposition is equal to
		\begin{equation*}
			(r,c)+\theta_1'(r_1+r_2, c_1+r_1\cdot c_2)=\theta_2'(r+(r_1+r_2), c+r\cdot(c_1+r_1\cdot c_2))
		\end{equation*}
		where 
		\begin{equation*}
			\theta_1'\colon (\partial(c_1)+(r_1+(\partial(c_2)-r_1)))+(r_1+r_2)\rightarrow(\partial(c_1)+r_1)+(\partial(c_2)+r_2)
		\end{equation*}
		and
		\begin{equation*}
			\theta_2'\colon (\partial(c)+(r+(\partial(c_1)+((r_1+\partial(c_2))-r_1))-r))+((r+r_1)+r_2)\rightarrow(\partial(c)+r)+((\partial(c_1)+r_1)+(\partial(c_2)+r_2)),
		\end{equation*}
		are the obvious special isomorphisms.
		
		We have the weak special congruence $(c+r\cdot c_1)+(r+r_1)\cdot c_2 \sim c+r\cdot (c_1+r_1\cdot c_2)$ and a special congruence $(r+r_1)+r_2\sim r+ (r_1+r_2)$, which give a commutative diagram like in the proof of Proposition \ref{prop:46}
		
		\begin{equation*}
			\xymatrix@C=65mm@R=15mm{
				(r+r_1)+r_2 \ar[d]_-{\varepsilon} \ar[r]^-{\theta_2((r+r_1)+r_2,(c+r\cdot c_1)+(r+r_1)\cdot c_2)} &  ((\partial(c)+r)+(\partial(c_1)+r_1))+(\partial(c_2)+r_2) \ar[d]^-{\varepsilon_1} \\
				r+(r_1+r_2) \ar[r]_-{\theta_2'(r+(r_1+r_2),c+r\cdot(c_1+r_1\cdot c_2))}  & (\partial(c)+r)+((\partial(c_1)+r_1)+(\partial(c_2)+r_2))
			}
		\end{equation*}
		where $\varepsilon$ and $\varepsilon_1$ are special isomorphisms. These prove the desired isomorphism of arrows.
	\end{proof}
	
	In more general case of arrows, the proposition can be proved in a similar way. Therefore we have proved that the sum of arrows in $\mathbb{G}_1$ is associative up to isomorphism. 
	
	The sum of objects in $\mathbb{G}_0$ is defined as a sum in c-group $N$, which is associative up to special congruence, and therefore it is associative in $\mathbb{G}_0$ up to special isomorphism in $\mathbb{G}_1$.
	
	A zero object in the category $\mathbb{G}$ with objects 
 $\mathbb{G}_0=N$ is a zero element 0 in $N$ as in c-group, i.e. we have special isomorphisms $0+r\approx r\approx r+0$ which are defined by special congruences in $N$.
	
	A zero element for the addition operation in $\mathbb{G}_1$ we define as $\sigma(0,0)$, where $\sigma$ is the special isomorphism defined by the special congruence $\sigma\colon \partial(0)+0\sim 0$. According to the definition of identity arrows for any object $r\in\mathbb{G}_0$, we have $\sigma(0,0)=1_0\colon 0\rightarrow 0$.
	
	\begin{prop}\label{prop:zeromorph}
		For any arrow $\beta(r,c)\alpha\in\mathbb{G}_1$ we have an isomorphism
		\begin{equation*}
			\beta(r,c)\alpha+\sigma(0,0)\approx  \beta(r,c)\alpha \approx \sigma(0,0)+\beta(r,c)\alpha
		\end{equation*}
	\end{prop} 
	\begin{proof}
		Let $\alpha\colon r'\rightarrow r$ and $\beta\colon \partial(c)+r\rightarrow r''$ be special isomorphisms. For the left-side of the isomorphism we have 
		\begin{equation*}
			\beta(r,c)\alpha+\sigma(0,0)=(\beta+\sigma)\zeta(r+0,c+r\cdot 0)(\alpha+1_0),
		\end{equation*}
		where $\beta+\sigma\colon (\partial(c)+r)+(\partial(0)+0)\rightarrow r''+0$ and $\alpha+1_0\colon r'+0\rightarrow r+0$ are obvious special isomorphisms. The picture is
		\begin{equation*}
			\xymatrix@C=15mm{
				r' \ar[r]^-{\alpha} & r \ar[r]^-{(r,c)} & \partial(c)+r \ar[r]^-{\beta} & r''
			}
		\end{equation*}
	
		\begin{equation*}
			\xymatrix@C=13mm{
				r'+0 \ar[r]^-{\alpha+1_0} & r+0 \ar[r]^-{(r+0,c+r\cdot 0)} & (\partial(c)+(r+(\partial(0)-r)))+(z+0) \ar[r]^-{\zeta} & (\partial(c)+r)+(\partial(0)+0)  \ar[r]^-{\beta+\sigma} & r''+0,
			}
		\end{equation*}
		where $\zeta$ is a special isomorphism. Moreover we have special isomorphisms $r\approx r+0$, $r'\approx r'+0$, $r''\approx r''+0$ and $c\approx c+r\cdot 0$. Hence by coherence of special isomorphisms and Proposition \ref{prop:44}, the isomorphism $\beta(r,c)\alpha+\sigma(0,0)\approx \beta(r,c)\alpha$ follows. The second isomorphism is proved in a similar way.
	\end{proof}
	
	For any arrow $\beta(r,c)\alpha\in\mathbb{G}_1$ we define an opposite arrows $-(\beta(r,c)\alpha)=(-\beta)\varepsilon(-r,(-r)\cdot (-c))(-\alpha)$, where $\alpha\colon r'\rightarrow r$ and $\beta\colon \partial(c)+r\rightarrow r''$ are special isomorphisms. $-\alpha\colon -r'\rightarrow -r$ and $-\beta\colon -(\partial(c)+r)\rightarrow -r''$ are the opposite special isomorphisms  corresponding to the opposite special congruences $-\alpha\colon -r'\sim  -r$ and $-\beta\colon -(\partial(c)+r)\sim -r''$, respectively; and $\varepsilon\colon ((-r)+(\partial(-c)-(-r)))+(-r)\rightarrow -(\partial(c)+r)$ is an obvious special isomorphism.
	
	\begin{prop}\label{prop:oppoarrow}
		For any $\beta(r,c)\alpha\in\mathbb{G}_1$ we have isomorphisms
		\begin{equation*}
			\beta(r,c)\alpha+(-(\beta(r,c)\alpha))\approx \sigma(0,0)\approx -(\beta(r,c)\alpha)+\beta(r,c)\alpha
		\end{equation*}
		where $\sigma\colon \partial(0)+0\rightarrow 0$ is the special isomorphism.
	\end{prop}
	\begin{proof}
		We have 
		\begin{equation}\label{eq:beta}
			\begin{alignedat}{2}
				\beta(r,c)\alpha + (-(\beta(r,c)\alpha)) & = \beta(r,c)\alpha + (-\beta)\varepsilon(-r,(-r)\cdot(-c))(-\alpha) \\
				& = (\beta+ ((-\beta)\varepsilon))\theta(r+(-r),c+r\cdot((-r)\cdot(-c)))(\alpha+(-\alpha)),
			\end{alignedat}
		\end{equation}
		where $\beta+ ((-\beta)\varepsilon)\colon (\partial(c)+r)+(((-r)+(\partial(-c)-(-r)))+(-r))\rightarrow r''-r''$ is a special isomorphism. By Proposition \ref{prop:44} we obtain that the right-side of \eqref{eq:beta} is isomorphic to $(0,0)$ and therefore to $\sigma(0,0)$. The second isomorphism of the proposition is proved similiarly.
	\end{proof}
	
	Below we prove that an opposite arrow does not depend on the choice of representatives.
	
	\begin{prop}
		For any special isomorphism $\alpha\in\mathbb{G}_1$, where $\alpha\equiv\varepsilon(r,c)$ by \textbf{Identification I}, we have an equality $-\alpha=(\varepsilon(r,c))$, where $\alpha\colon r\rightarrow r'$ and $\varepsilon\colon (r'-r)+r\rightarrow r'$ are special isomorphisms and $\partial(c)=r'-r$. 
	\end{prop}
	\begin{proof}
		By the definition of opposite arrow we have $-(\varepsilon(r,c))=(-\varepsilon)\varphi(-r,(-r)\cdot(-c))$,
		where 
		\begin{equation*}
			\varphi\colon (-r+(\partial(-c)-(-r)))+(-r)\colon -r+(r-r')\rightarrow -r',
		\end{equation*}
		is a special isomorphism. Here we applied that $\partial(-c)\sim -\partial(c)$ in $N$, and therefore $\partial(-c)\approx-\partial(c)$ in $\mathbb{G}_1$. $-\alpha\colon -r \sim -r'$ is a special congruence, therefore we have $-\alpha\equiv\varepsilon'(-r,c')$, where $\partial(c')=-r'-(-r)$ and $\varepsilon'\colon (-r'-(-r))+(-r)\rightarrow -r'$ is a special isomorphism. We have to show that
		\begin{equation}\label{eq:varphi}
			\varepsilon'(-r,c')=(-\varepsilon)\varphi(-r,-r\cdot (-c)).
		\end{equation}
		$c'$ and $-r\cdot (-c)$ are weakly special congruent to 0; from the coherence of special isomorphisms and \textbf{Identification II} we prove equality \eqref{eq:varphi}.
	\end{proof}
	
	\begin{prop}
		If $\varphi\colon c\sim c'$ is a weak special congruence, then 
		\begin{equation*}
			-(\beta'(r,c')\alpha)=-(\beta(r,c)\alpha),
		\end{equation*}
		for any $r\in\mathbb{G}_0$ (see diagram \eqref{diag:31}).
	\end{prop}
	\begin{proof}
		By \textbf{Identification II} we have $\beta'(r,c')\alpha=\beta(r,c)\alpha$. The proof is similar to the previous statement, applying the definition of opposite arrows and \textbf{Identification II}.
	\end{proof}
	
	\begin{thm}
		The category $\mathbb{G}=(\mathbb{G}_0,\mathbb{G}_1,d_0,d_1,i,m)$ with the defined addition operation in $\mathbb{G}$ is a categorical group.
	\end{thm}
	\begin{proof}
		First of all we prove that the addition operation in $\mathbb{G}$ is a functor $\mathbb{G}\times \mathbb{G}\rightarrow \mathbb{G}$. Consider the following picture of objects and arrows in $\mathbb{G}\times \mathbb{G}$.
		\begin{equation*}
			\xymatrix@R=13mm{
	r \ar[d]^-{(r,c)}	& r' \ar[d]^-{(r',c')}	& &	(r,r') \ar[d]^-{((r,c),(r',c'))}\\
	\partial(c)+r \ar[d]^-{(\partial(c)+r,c_1)}	& \partial(c')+r' \ar[d]^-{(\partial(c')+r',c_1')}	& \Rightarrow &	(\partial(c)+r,\partial(c')+r') \ar[d]^-{((\partial(c)+r,c_1),(\partial(c')+r',c_1'))}\\
	\partial(c_1)+(\partial(c)+r)	& \partial(c_1')+(\partial(c')+r')	& &	(\partial(c_1)+(\partial(c)+r),\partial(c_1')+(\partial(c')+r'))
			}
		\end{equation*}
		Here we have in mind that on the base of definition of sum for arrows we have
		\begin{equation*}
			d_j((r,c)+(r',c'))=d_j(r,c)+d_j(r',c'), \text{ for } j=0,1.
		\end{equation*}
		Let us first take the (vertical) composition of arrows and then the sum of resulting compositions. Then we have
		\begin{equation*}
			(\varepsilon(r,c_1+c),\varepsilon'(r',c_1'+c'))\longrightarrow \xi(r+r',(c_1+c)+r\cdot(c_1'+c'))
		\end{equation*}
		where $\varepsilon\colon (\partial(c_1)+\partial(c))+r_1\approx \partial(c_1)+(\partial(c)+r)$ and $\varepsilon'\colon (\partial(c_1')+\partial(c'))+r_1'\approx \partial(c_1)+(\partial(c')+r')$ are the special isomorphisms in $\mathbb{G}$. We have the following equality  and the special isomorphism
		\begin{alignat*}{2}
			\partial((c_1+c)+r\cdot (c_1'+c'))+(r+r')&=(\partial(c_1)+\partial(c))+((r+(\partial(c_1')-r))+(r+(\partial(c')-r)))+(r+r') \\
			&\stackrel{\xi}{\approx} (\partial(c_1)+(\partial(c)+r))+(\partial(c')+r').
		\end{alignat*}
	Now let us take first sums of arrows and then the composition of the sums. We have
		\begin{equation*}
			(r,c)+(r',c')=\eta(r+r',c+r\cdot c'),
		\end{equation*}
		where $\eta$ is the special isomorphism  $\partial(c)+((r+\partial(c'))-r)+(r+r')\approx(\partial(c)+r)+(\partial(c')+r')$.
		
		For the second sum we have
		\begin{equation*}
			(\partial(c)+r,c_1)+(\partial(c')+r',c_1')=\xi'((\partial(c)+r)+(\partial(c')+r'),c_1+((\partial(c)+r)\cdot c_1'))
		\end{equation*}
		where $\xi'$ is a special isomorphism
		\begin{equation*}
			(\partial(c_1)+((\partial(c)+r)+(\partial(c')-(\partial(c)+r))))+((\partial(c)+r)+(\partial(c')+r'))\approx (\partial(c_1)+(\partial(c)+r))+(\partial(c_1')+(\partial(c')+r')).
		\end{equation*}
		
		Now we take the composition of these sums
		\begin{equation*}
			\xi'((\partial(c)+r)+(\partial(c')+r'),c_1+(\partial(c)+r)\cdot c_1')\circ\eta(r+r',c+r\cdot c')=\chi(r+r',(c_1+(\partial(c)+r)\cdot c_1')+(c+r\cdot c')),
		\end{equation*}
		where $\chi$ is the special isomorphism
		\begin{alignat*}{2}
		&(\partial((c_1+(\partial(c)+r)\cdot c_1)+\partial(c+r\cdot c')))+(r+r')\\
		&	\phantom{aaaa}=((\partial(c_1)+((\partial(c)+r)+(\partial(c_1')-(\partial(c)+r))))+(\partial(c)+(r+(\partial(c')-r))))+(r+r')\\
		&\phantom{aaaa}\stackrel{\chi}{\approx} (\partial(c_1)+(\partial(c)+r))+(\partial(c_1')+(\partial(c')+r')).
		\end{alignat*}
		We have to compare the arrows $\xi(r+r',(c_1+c)+r\cdot(c_1'+c'))$ and $\chi(r+r',(c_1+(\partial(c)+r)\cdot c_1')+(c+r\cdot c'))$. We have weak special congruence
		\begin{equation*}
			(c_1+c)+r\cdot (c_1'+c')\sim (c_1+(\partial(c)+r)\cdot c_1')+(c+r\cdot c').
		\end{equation*}
		On the base of \textbf{Identification II} and the coherence for special isomorphisms we have the equality of these two arrows. 
		
		Now we show that for any object $(r,r')$ in $\mathbb{G}\times \mathbb{G}$, the identity arrow is carried to the identity arrow of the object $r+r'$ in $\mathbb{G}_0$. We have to show that
		\begin{equation}\label{eq:epsilon}
			\varepsilon(r,0)+\varepsilon'(r',0)=\varepsilon''(r+r',0)
		\end{equation}
		where $\varepsilon\colon \partial(0)+r\rightarrow r$, $\varepsilon'\colon \partial(0)+r'\rightarrow r'$ and $\varepsilon''\colon \partial(0)+(r+r')\rightarrow r+r'$ are the special isomorphisms. For the left-side we have 
		\begin{equation*}
			\varepsilon(r,0)+\varepsilon'(r',0)=\zeta(r+r',0+r\cdot0)
		\end{equation*}
		where $\zeta\colon \partial(0+r\cdot 0)+(r+r')\rightarrow r+r'$ is the special isomorphism. We have weak special congruence $0+r\cdot 0\sim 0$. By \textbf{Identification I} we obtain equality \eqref{eq:epsilon}.
		
		We have to show that special isomorphisms are natural in $\mathbb{G}$.
		
		We have a zero arrow in $\mathbb{G}$, $1_0=\sigma(0,0)$, where $\sigma\colon \partial(0)+0\approx 0$ is a special isomorphism. By Proposition \ref{prop:zeromorph} for any $\beta(r,c)\alpha\in\mathbb{G}_1$ we have $\beta(r,c)\alpha+1_0\approx \beta(r,c)\alpha$, from which the naturality of the special isomorphism $r+0\approx r$ follows as in the diagram
		\begin{equation*}
			\xymatrix@R=15mm{
				r+0 \ \approx \ r  \ar@<-4.5mm>[d]_-{\beta(r,c)\alpha+1_0} \ar@<7.3mm>[d]^-{\beta(r,c)\alpha}\\
				r'+0 \ \approx \ r' 
			}
		\end{equation*} 
		
		In a similar way, on the base of the definition of sum of arrows in $\mathbb{G}_1$, and associativity of the sum up to isomorphism (Proposition \ref{prop:assoc}) we obtain that the special isomorphism $(r_1+r_2)+r_3\approx r_1+(r_2+r_3)$ is natural. The same is true for the special isomorphism $r+(-r)\approx 0$, from the definition of opposite arrows in $\mathbb{G}_1$ and the isomorphism $\beta(r,c)\alpha+(-(\beta(r,c)\alpha))\approx\sigma(0,0)$ (Proposition \ref{prop:oppoarrow}). 
		
		At the end we note once more that $\mathbb{G}$ is a coherent categorical group. By the fact that we have noted several times in the proofs of the paper, special isomorphisms in $\mathbb{G}$ are special congruences in the corresponding c-group which satisfy coherence condition.
		
		It is worth to note that from the functorial property of addition in $\mathbb{G}$, the interchange law for the arrows in $\mathbb{G}_1$ follows.
	\end{proof}

	% \bibliographystyle{elsarticle-num} 
	% \bibliography{cas-refs}
\end{document}